\documentclass{article}
\usepackage{arxiv}
\pdfoutput=1

\usepackage{graphicx}
\usepackage{epstopdf}
\DeclareGraphicsExtensions{.eps,.png,.jpg,.pdf}
\usepackage{float}
\usepackage{subfig} 
\graphicspath{{./img/}{./}}
    
\usepackage{amsmath}
\usepackage{amsthm} 
\usepackage{amssymb}
\usepackage{mathtools}
\usepackage{mathrsfs}
\allowdisplaybreaks[1] 
\usepackage[latin1]{inputenc} 
\usepackage{amsfonts}
\usepackage{textcomp}

\usepackage{xcolor,colortbl} 
\definecolor{LightCyan}{rgb}{0.70,1,1}

\usepackage{multirow}
\usepackage{booktabs}
\usepackage{tabulary}

\usepackage{hyperref}
\usepackage{comment}

\usepackage{caption}
\usepackage{adjustbox}
     
\usepackage{algorithm, algpseudocodex, setspace}

\usepackage[shortlabels]{enumitem}

\usepackage{multicol} 

\newtheorem{definition}{Definition}
\newtheorem{theorem}{Theorem}

\newtheorem{remark}{Remark}

\newtheorem{proposition}{Proposition}

\theoremstyle{remark} 

\newcommand{\argmin}[1]{\mathchoice
    {\operatorname{arg}\, \underset{#1}{\operatorname{min}}\;}%
    {\operatorname{arg}\, \operatorname{min}_{#1}\;}%
    {\operatorname{arg}\, \operatorname{min}_{#1}\;}%
    {\operatorname{arg}\, \operatorname{min}_{#1}\;}%
}
\newcommand{\argmax}[1]{\mathchoice
    {\operatorname{arg}\, \underset{#1}{\operatorname{max}}\;}%
    {\operatorname{arg}\, \operatorname{max}_{#1}\;}%
    {\operatorname{arg}\, \operatorname{max}_{#1}\;}%
    {\operatorname{arg}\, \operatorname{max}_{#1}\;}%
}

\newcommand{\euclideannorm}[1]{ \left\Vert {#1} \right\Vert_2}

\newcommand{\quotes}[1]{``{#1}''}

\begin{document}

\title{A unified surrogate-based scheme for black-box and preference-based optimization}

\author{
    Davide Previtali\\
    Department of Computer Science Engineering\\
    University of Bergamo\\
    \texttt{davide.previtali@unibg.it}\\
    \And
    Mirko Mazzoleni\\
    Department of Computer Science Engineering\\
    University of Bergamo\\
    \texttt{mirko.mazzoleni@unibg.it}\\
    \And
    Antonio Ferramosca\\
    Department of Computer Science Engineering\\
    University of Bergamo\\
    \texttt{antonio.ferramosca@unibg.it}\\
    \And
    Fabio Previdi\\
    Department of Computer Science Engineering\\
    University of Bergamo\\
    \texttt{fabio.previdi@unibg.it}\\
}

\date{}

\maketitle

\begin{abstract}
    Black-box and preference-based optimization algorithms are global optimization procedures that aim to find the global solutions of an optimization problem using, respectively, the least amount of function evaluations or sample comparisons as possible. In the black-box case, the analytical expression of the objective function is unknown and it can only be evaluated through a (costly) computer simulation or an experiment. In the preference-based case, the objective function is still unknown but it corresponds to the subjective criterion of an individual. So, it is not possible to quantify such criterion in a reliable and consistent way. Therefore, preference-based optimization algorithms seek global solutions using only comparisons between couples of different samples, for which a human decision-maker indicates which of the two is preferred. Quite often, the black-box and preference-based frameworks are covered separately and are handled using different techniques. In this paper, we show that black-box and preference-based optimization problems are closely related and can be solved using the same family of approaches, namely surrogate-based methods. Moreover, we propose the generalized Metric Response Surface (\texttt{gMRS}) algorithm, an optimization scheme that is a generalization of the popular \texttt{MSRS} framework. Finally, we provide a convergence proof for the proposed optimization method.
\end{abstract}

\keywords{Global optimization, Black-box optimization, Preference-based optimization, Bayesian optimization.}

\section{Introduction}
\label{sec:Introduction}
In many applications there is the need to find the ``optimal'' value for a {decision variable}, i.e. the one that maximizes a measure of performance, minimizes some cost or best satisfies a human decision-maker's criterion. For instance, in the context of control systems, we might be interested in tuning the parameters of a controller to achieve some desired performance \cite{domanski2020control}.
% Similarly, many machine learning methods require a proper calibration of their hyperparameters so that certain requirements are met.
However, in some cases, it might be impossible to objectively quantify the \quotes{goodness} of a certain decision variable. For instance, an evaluation of a controller performance might depend on a human operator, that expresses a judgement through visual inspection (or other sensory evaluations) of the behavior achieved by the system under control.
% This situation is typical in many manufacturing industries, where the expertise of how to tune a machine is often in the hands of few experts, and a shared, quantifiable tuning procedure is not available. These situations are not envisaged since: (i) they make the company highly dependent on these expert people, and (ii) the human-guided procedure is error prone and not repeatable.
%
% Whether the \quotes{goodness} of a certain decision variable is evaluated
% objectively or subjectively, we are often interested in finding its ``best value'', i.e. the one that maximizes a measure of performance, minimizes some cost or best satisfies the human decision-maker's criterion. 
%
These optimization problems can be stated as: \textit{find the global solution}\footnote{In general, an optimization problem can have multiple global solutions. Here, we consider the case where only one global solution is present. We do not make any assumptions on the local optimizers, which can be more than one.} \textit{of an optimization problem} whose \textit{objective function} can either be: (i) completely known (i.e. its analytical expression is available), (ii) unknown but measurable or (iii) unknown and not objectively quantifiable. Further complications arise if the evaluation of the the objective function is \textit{expensive}, i.e. a non-negligible amount of resources needs to be spent to asses the \quotes{goodness} of a decision variable (for instance, its measure might require running a time-expensive computer simulation or performing experiments on a real system).
Depending on (i) the \textit{knowledge available} on the objective function, as well as (ii) \textit{how easy it is to acquire information} on it, different optimization frameworks should be employed, see Figure \ref{fig:Summary_optimization_frameworks}.
\begin{figure*}[!ht]
    \centering      \includegraphics[width=\textwidth]{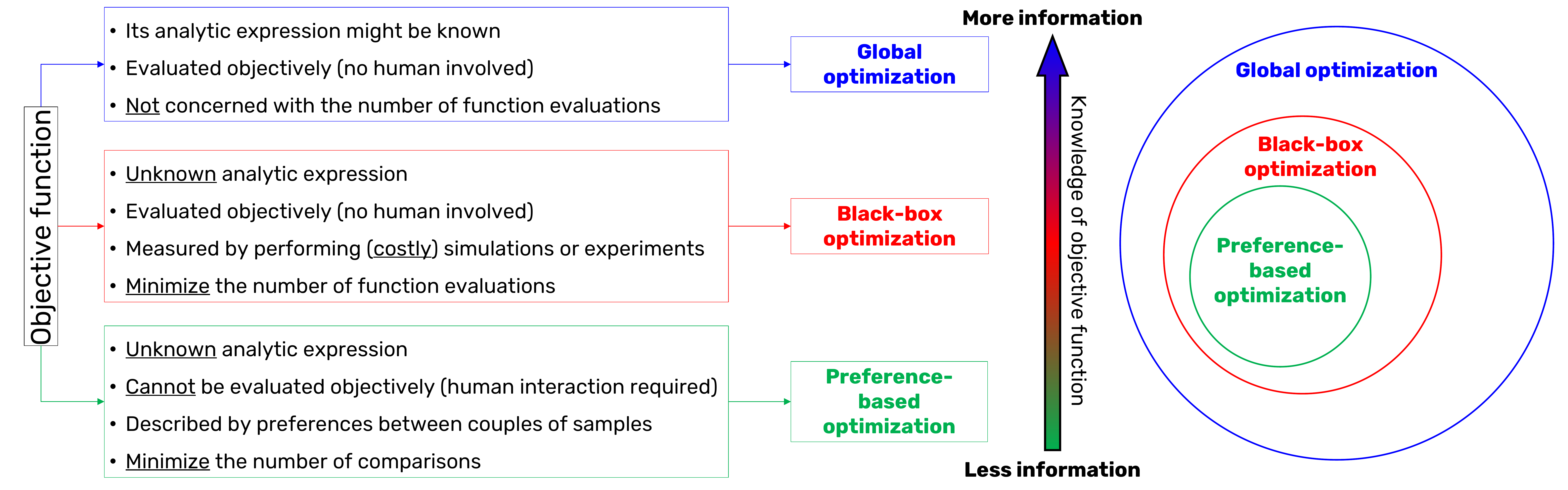}
    \caption{\label{fig:Summary_optimization_frameworks} Summary of the considered optimization frameworks.}
\end{figure*}

% %(for example, by extensively sampling the whole domain of the decision variable), 
% it is best to use algorithms that pertain to \textit{different
% %(but closely related) 
% optimization frameworks}. 
%
% In this paper, we will cover the global, black-box and preference-based optimization frameworks.
%

% Global optimization
Whenever the objective function is known %(i.e. its mathematical formulation is available) 
or is quite cheap to evaluate, it is best to employ \textit{global optimization} techniques, that can either be derivative-based \cite{nocedal2006numerical} or derivative-free \cite{rios2013derivative}. In the first case, it is possible to combine a derivative-based local search algorithm with a multi-start method \cite{Marti2018} to reach the global solution. Instead, derivative-free techniques are quite useful whenever the objective function is not differentiable or if the derivative information is unreliable (e.g. if it is obtained by finite differentiation of noisy measures). Some popular derivative-free algorithms are DIvide a hyper-RECTangle (\texttt{DIRECT}) \cite{jones1993lipschitzian}, Particle Swarm Optimization (\texttt{PSWARM}) \cite{vaz2007particle} and evolutionary algorithms \cite{hansen2001completely}.

% Black-box optimization
The main drawback of the aforementioned techniques is the often excessive number of function evaluations required to find the global solution. This could be quite prohibitive when the objective function is \textit{unknown} and \textit{expensive} to measure.
% (or, more properly, it is a \textit{black-box}, i.e. no analytical expression is available)
%
When that is the case, a better suited class of algorithms are \textit{black-box optimization} techniques \cite{audet2017derivative}, which aim to both minimize the number of function evaluations and obtain the global optimizer.
A family of procedures within such framework is called \textit{surrogate-based} (or \textit{surface response}) \textit{methods}. These algorithms aim to both approximate the unknown objective function, using a so-called \textit{surrogate model}, and explore the domain of the decision variable sufficiently enough to converge to the global solution. In practice, such methods iteratively propose new samples to be evaluated by properly trading-off \textit{exploitation} (local search)
%
%, i.e. selecting samples that are likely to offer an improvement based on the current observations, 
%
and \textit{exploration} (global search).
%, i.e. finding samples in regions of the domain of which we have little to no knowledge (where the surrogate model is most uncertain). 
%
This is done by defining a suitable \textit{acquisition function}
% , which encapsulates these two aspects,
and the next candidate sample is obtained by minimizing or maximizing it.
%
% Throughout the years, many surrogate-based methods have been developed; these mostly differ on the surrogate model employed and on the definition of the acquisition function. 
Some good and extensive surveys on the topic are \cite{vu2017surrogate, jones2001taxonomy}. The most popular surface response methods either approximate the black-box function using Gaussian Processes, giving rise to Bayesian Optimization \cite{brochu2010tutorial}, or through Radial Basis Functions, see for example the algorithm proposed by Gutmann (\texttt{Gutmann-RBF}) \cite{gutmann2001radial}, Constrained Optimization using Response Surfaces (\texttt{CORS}) \cite{regis2005constrained}, Metric Stochastic Response Surface Method (\texttt{MSRS}) \cite{Regis2007} and the more recent GLobal minimum using Inverse distance weighting and Surrogate radial basis functions (\texttt{GLIS}) \cite{Bemporad2020}.

% (Active) preference-based optimization
When the objective function can only be \textit{evaluated subjectively}, or rather it describes a human decision-maker's criterion that cannot be expressed analytically, a possible way to solve the optimization problem consists of iteratively asking the user to compare couples of different samples, expressing \textit{preferences} between them. All the information that concerns the tastes of an individual is encapsulated in a \textit{preference relation}, which describes the outcomes of the comparisons.
% 
% 
% 
% rate (and thus somehow quantify) samples of the decision variable and find the global solution based on this information, using the previously seen techniques.
% 
% Suppose now that the objective function can only be evaluated subjectively, or rather that it describes a human decision-maker's criterion which cannot be expressed analytically . Under this hypothesis, a possible way to solve the optimization problem consists of asking the user to rate (and thus somehow quantify) samples of the decision variable and find
% the global solution based on this information, using the previously seen techniques.
% However, as pointed out in \cite{brochu2007active}, it is quite difficult to get a consistent evaluation. Instead, a more reliable way to acquire information about the human's criterion is to ask him/her to compare two different samples and express a \textit{preference} between them.
% Formally, all the information that concerns the tastes of the individual
% is embedded in a binary relation called the \textit{preference relation}.
There exist many fundamental results in \textit{utility theory} that, under some hypotheses, allow us to \textit{represent} the preference relation with a \textit{(latent) utility function} \cite{ok2011real}, i.e. a function that assigns an abstract degree of \quotes{goodness} to all possible values of the decision variable. In this case, the best sample for a human decision-maker is the one that has the highest utility. To find the maximizer of the utility function, it is possible to use \textit{(active) preference-based optimization} algorithms (sometimes referred to as active preference learning\footnote{We want to make a clear distinction between preference learning and preference-based optimization. The former aims to approximate the latent utility function \cite{furnkranz2010preference} with a predictive model, as commonly done in machine learning. Instead, the latter aims to find the global optimizer of an optimization problem using only the information brought by the preferences. In practice, many preference-based optimization methods still use a predictive model, yet its prediction accuracy is not the main concern.}), which also aim to minimize the number of pairwise comparisons. Surface response methods for preference-based optimization build a surrogate model for the latent utility function using the preferences expressed by the individual. Similarly to the black-box case, a suitable acquisition function needs to be defined in order to find the next candidate sample to evaluate.
% The keyword \quotes{active} refers to the fact that a human decision-maker is iteratively asked to express a preference between two samples.
% In practice, many preference-based optimization algorithms build surrogate models for either or both the preference relation and the utility function, using the expressed preferences. 
% 
% As for black-box optimization, a suitable acquisition function needs to be defined in order to find the next candidate sample to evaluate. Then, usually, the user is asked to compare this sample with the best one found so far by the procedure, obtaining a new preference.
% 
% 
Most preference-based optimization algorithms are extensions of Bayesian Optimization, see for example \cite{brochu2007active, gonzalez2017preferential, benavoli2021preferential}. Quite recently, the authors of \cite{Bemporad2020} proposed an extension of \texttt{GLIS}
in the preference-based framework, called \texttt{GLISp} \cite{Bemporad2021}, that is based on a radial basis function surrogate.

Global, black-box and preference-based optimization  are often treated separately in the literature. Moreover, a unified view for the resolution of these optimization problems has not yet been proposed. In this paper, we show how black-box and preference-based frameworks can be seen as \textit{particular cases of global optimization}, since they all aim to find the global solution of an optimization problem.
% However, they forego an higher degree of optimality accuracy in order to minimize the number of function evaluations or sample comparisons respectively. 
At the same time, \textit{preference-based optimization can be interpreted as an instance of black-box optimization}, where the objective function is both unknown (black-box) and cannot be measured explicitly.
% (i.e. we cannot obtain its values in a consistent fashion, but we must resort to sample comparisons).
Considering preference-based optimization as a specific case of black-box optimization can ease the definition of new algorithms for the former framework. Moreover, results and techniques applied for black-box procedures can be carried over to preference-based ones.
% For instance, in this work we employ a cycling strategy for the exploration-exploitation trade-off that has been adopted in the \texttt{MSRS} algorithm. Finally, similarly to  \texttt{MSRS} and  \texttt{CORS} we separate the choice of the surrogate model (which can be problem specific) from the optimization strategy.
%
%
The main contributions of this work are:
\begin{enumerate}
    \item Provide a thorough \textit{comparison of black-box and preference-based optimization}, highlighting key similarities and differences, and show that, from an utility theory perspective, they both aim to solve the same optimization problem;
    \item Propose a \textit{general surrogate-based optimization scheme} that can be applied to both black-box and preference-based frameworks;
    \item Provide a \textit{proof of convergence} for such surrogate-based scheme. Notably, it is possible to prove the convergence in the preference-based case by leveraging results from the global optimization literature and the utility theory framework.
\end{enumerate}

The paper is organized as follows. Section \ref{sec:Problems_formulation}  introduces and compares the black-box and preference-based optimization problems. Section \ref{sec:Surrogate_models} describes two popular surrogate models, based on Radial Basis Functions and Gaussian Processes. Section \ref{sec:Handling_exploration_and_exploitation} proposes an acquisition function suited for both black-box and preference-based optimization, while
%based on an explicit trade-off between the surrogate model and an exploration function. Concerning the latter, we define what it means for an exploration function to be \quotes{proper}, a property that is linked with the convergence of the proposed surrogate-based scheme shown in Section \ref{sec:General_optimization_scheme}.  
Section \ref{sec:General_optimization_scheme} provides a general surrogate-based optimization scheme, based on the proposed acquisition function. Its convergence is proven both in the black-box and preference-based frameworks. An example of the proposed optimization scheme is shown in Section \ref{sec:example}.
Finally, Section \ref{sec:Summary_and_discussion} is devoted to concluding remarks.
\section{Problems formulation}
\label{sec:Problems_formulation}
In this Section we are going to compare the black-box and preference-based optimization frameworks, showing how they both solve the same optimization problem using different information on the objective function.

\subsection{Black-box optimization}
The aim of black-box optimization is to solve the following global optimization problem:
\begin{align}
    \label{eq:black_box_optimization_problem_no_constraints}
    \boldsymbol{x}^{\boldsymbol{*}} & = \argmin{\boldsymbol{x}} f(\boldsymbol{x}) \\
    \text{s.t.}                     & \quad\boldsymbol{x}\in\Omega, \nonumber
\end{align}
where $\boldsymbol{x} = \begin{bmatrix}x^{(1)} & \ldots & x^{(n)}\end{bmatrix}^{\top} \in\mathbb{R}^{n}$ is the {decision variable}, $f:\mathbb{R}^{n}\to\mathbb{R}$ is a black-box cost function (unknown and expensive to evaluate) and $\Omega \subset \mathbb{R}^{n}$ is the \textit{constraint set} which, in its most general formulation, is given by
\begin{align}
    \label{eq:constraint_set_Omega}
    \Omega=\Bigl\{\boldsymbol{x}:\  & \boldsymbol{l}\leq\boldsymbol{x}\leq\boldsymbol{u},                          &  & \text{bounds} \nonumber                 \\
                                    & A_{ineq}\cdot\boldsymbol{x}\leq\boldsymbol{b_{ineq}},                        &  & \text{linear inequalities} \nonumber    \\
                                    & A_{eq}\cdot\boldsymbol{x}=\boldsymbol{b_{eq}},                               &  & \text{linear equalities}     \nonumber  \\
                                    & \boldsymbol{g_{ineq}}(\boldsymbol{x})\leq\boldsymbol{0}_{p_{ineq}}, &  & \text{nonlinear inequalities} \nonumber \\
                                    & \boldsymbol{g_{eq}}(\boldsymbol{x})=\boldsymbol{0}_{p_{eq}}\Bigr\}. &  & \text{nonlinear equalities}
\end{align}
\begin{comment}
\begin{align}
    \Omega = \Big\{ & \boldsymbol{x}:
    \boldsymbol{l}\leq\boldsymbol{x}\leq\boldsymbol{u},
    A_{ineq}\cdot\boldsymbol{x}\leq\boldsymbol{b_{ineq}},
    A_{eq}\cdot\boldsymbol{x}=\boldsymbol{b_{eq}},  \nonumber                  \\
                    & \boldsymbol{g_{ineq}}(\boldsymbol{x})\leq\boldsymbol{0},
    \boldsymbol{g_{eq}}(\boldsymbol{x})=\boldsymbol{0}
    \Big\}.
\end{align}
\end{comment}
In \eqref{eq:constraint_set_Omega}, $\boldsymbol{l}, \boldsymbol{u} \in \mathbb{R}^{n}$, $A_{ineq} \in \mathbb{R}^{q_{ineq} \times n}$, $\boldsymbol{b_{ineq}} \in \mathbb{R}^{q_{ineq}}$, $A_{eq} \in \mathbb{R}^{q_{eq} \times n}$, $\boldsymbol{b_{eq}} \in \mathbb{R}^{q_{eq}}$, $\boldsymbol{g_{ineq}}: \mathbb{R}^{n} \to \mathbb{R}^{p_{ineq}}$ and $\boldsymbol{g_{eq}}: \mathbb{R}^{n} \to \mathbb{R}^{p_{eq}}$. Notation-wise, $\boldsymbol{0}_{p_{ineq}}$ represents the ${p_{ineq}}$ zero column vector (and similarly for $\boldsymbol{0}_{p_{eq}}$). We suppose that: (i) all of these constraints are completely known and (ii) Problem \eqref{eq:black_box_optimization_problem_no_constraints} is, at least, bound constrained. If $\Omega$ is compact and $f\left(\boldsymbol{x}\right)$ is continuous, then Problem \eqref{eq:black_box_optimization_problem_no_constraints} admits a solution according to the Extreme Value Theorem \cite{audet2017derivative}.

Surrogate-based methods solve Problem \eqref{eq:black_box_optimization_problem_no_constraints} starting from a set $\mathcal{X}$ of $N$ \textit{distinct samples} of the decision variable, defined as:
\begin{equation}
    \label{eq:sample_set_X}
    \mathcal{X}=\left\{ \boldsymbol{x}_{i}:i=1,\ldots,N,\boldsymbol{x}_{i}\in\Omega,
    \boldsymbol{x}_{i} \neq \boldsymbol{x}_{j}, \forall i \neq j\right\},
\end{equation}
as well as the corresponding values assumed by the cost function at those samples.
In practice, the measure of $f\left(\boldsymbol{x}\right)$ could be affected by {noise}, which is assumed to be a zero-mean Gaussian white noise with variance $\sigma_{\eta}^2$. We define the \textit{set of measures} as:
\begin{equation}
    \label{eq:function_measure_set_Y}
    \mathcal{Y}=
    \left\{
    y_{i}:y_{i}=f\left(\boldsymbol{x}_{i}\right) + \eta_i,\boldsymbol{x}_{i}\in\mathcal{X},
    \eta_i \overset{i.i.d.}{\sim} \mathcal{N}\left(0, \sigma^2_\eta\right)
    \right\}.
\end{equation}
The cardinality of sets $\mathcal{X}$ and $\mathcal{Y}$ is $\left|\mathcal{X}\right|=\left|\mathcal{Y}\right|=N$.

\subsection{Preference-based optimization}
\label{subsec:problem_formulation_preference-based}
In the preference-based framework, there is no function $f\left(\boldsymbol{x}\right)$ to be measured explicitly. Instead, a human decision-maker expresses his/her preferences between couples of samples. A fundamental question to ask is:

\textit{
    Can an arbitrary criterion of an individual be \quotes{translated} into a mathematical function $f\left(\boldsymbol{x}\right)$ such that solving Problem \eqref{eq:black_box_optimization_problem_no_constraints} leads to finding his/her most preferred value for the decision variable?
}

To answer such question, we will now give a brief overview of some important results in \textit{utility theory} \cite{ok2011real}, which allow us to formalize the preference-based optimization framework.
Consider the constraint set $\Omega$  in \eqref{eq:constraint_set_Omega}, we define a generic binary relation $\mathcal{R}$ on $\Omega$ as a subset  $\mathcal{R} \subseteq \Omega \times \Omega$. Notation-wise, given two samples $\boldsymbol{x}_{i}, \boldsymbol{x}_{j} \in \Omega$, we denote the ordered pairs for which the binary relation holds, $\left(\boldsymbol{x}_{i}, \boldsymbol{x}_{j}\right) \in \mathcal{R}$, as $\boldsymbol{x}_{i} \mathcal{R} \boldsymbol{x}_{j}$ \cite{ok2011real}.

A \textit{preference relation}, \mbox{$\succsim \subseteq \Omega \times \Omega$}, is a preorder (a specific case of binary relation) which is commonly used to describe the tastes of an individual. In this context, $\boldsymbol{x}_{i} \succsim \boldsymbol{x}_{j}$ implies that a human decision-maker with preference relation $\succsim$ deems sample $\boldsymbol{x}_{i}$ \textit{at least as good} as $\boldsymbol{x}_{j}$. The fact that the preference relation is a preorder encompasses the \textit{rationality} of the individual, since the following properties hold:
\begin{enumerate}
    \item \textit{Reflexivity}, i.e. $\boldsymbol{x}_{i} \succsim \boldsymbol{x}_{i}, \forall \boldsymbol{x}_{i} \in \Omega$ (any alternative is as good as itself),
    \item \textit{Transitivity}, i.e. $\forall \boldsymbol{x}_{i}, \boldsymbol{x}_{j}, \boldsymbol{x}_{k} \in \Omega$, if $\boldsymbol{x}_{i} \succsim \boldsymbol{x}_{j}$ and $\boldsymbol{x}_{j} \succsim \boldsymbol{x}_{k}$ hold, then $\boldsymbol{x}_{i} \succsim \boldsymbol{x}_{k}$ (consistency of the preferences expressed by the individual).
\end{enumerate}
The preference relation $\succsim$ is usually \quotes{split} into two transitive binary relations:
\begin{itemize}
    \item The \textit{strict preference relation} $\succ$ on $\Omega$, i.e.
          $\boldsymbol{x}_{i} \succ \boldsymbol{x}_{j}$ if and only if
          $\boldsymbol{x}_{i} \succsim \boldsymbol{x}_{j}$ but not
          $\boldsymbol{x}_{j} \succsim \boldsymbol{x}_{i}$
          ($\boldsymbol{x}_{i}$ is \quotes{better than} $\boldsymbol{x}_{j}$), and
    \item The \textit{indifference relation} $\sim$ on $\Omega$, i.e.
          $\boldsymbol{x}_{i} \sim \boldsymbol{x}_{j}$ if and only if
          $\boldsymbol{x}_{i} \succsim \boldsymbol{x}_{j}$ and
          $\boldsymbol{x}_{j} \succsim \boldsymbol{x}_{i}$
          ($\boldsymbol{x}_{i}$ is \quotes{as good as} $\boldsymbol{x}_{j}$).
\end{itemize}
Another common assumption on $\succsim$
% , 
% which is not necessarily related to the rationality of the individual, 
is that it is a \textit{complete} binary relation, i.e. either $\boldsymbol{x}_{i} \succsim \boldsymbol{x}_{j}$ or $\boldsymbol{x}_{j} \succsim \boldsymbol{x}_{i}$ hold $\forall \boldsymbol{x}_{i}, \boldsymbol{x}_{j} \in \Omega$. Completeness of $\succsim$ implies that the human decision-maker is never uncertain, that is he/she is always able to express a preference between any couple of samples. One last relevant property for $\succsim$ is \textit{continuity}. Here, we avoid a formal definition of the continuity of a binary relation \cite{ok2011real} but, intuitively, if $\succsim$ is continuous and $\boldsymbol{x}_i \succ \boldsymbol{x}_j$, then an alternative $\boldsymbol{x}_k$ which is \quotes{very close} to $\boldsymbol{x}_j$ should also be deemed strictly worse than $\boldsymbol{x}_i$.% by the individual.

Having defined the preference relation $\succsim$, the goal of preference-based optimization is to find the $\succsim$-maximum of $\Omega$, i.e. the sample $\boldsymbol{x}^* \in \Omega$ such that $\boldsymbol{x}^* \succsim \boldsymbol{x}, \forall \boldsymbol{x} \in \Omega$ (the most preferred by the individual).
Concerning the existence of $\boldsymbol{x}^*$, we can state the following Proposition, which can be seen as a generalization of the Extreme Value Theorem \cite{audet2017derivative} for preference relations.
\begin{proposition}[Existence of a $\succsim$-maximum of $\Omega$ \cite{ok2011real}]
    \label{prop:existence_of_preference_relation_maximum}
    A $\succsim$-maximum of  $\Omega$ is guaranteed to exist if $\Omega$ is a compact subset of a metric space (in our case $\Omega \subset \mathbb{R}^{n}$)   and $\succsim$ is a continuous and complete preference relation on $\Omega$.
\end{proposition}
Proposition \ref{prop:existence_of_preference_relation_maximum} allows us to prove the convergence of the proposed optimization scheme (Section \ref{sec:General_optimization_scheme}) in the preference-based case. One of the most important results in utility theory is the following Theorem.
\begin{theorem}[Debreu's Utility Representation Theorem for $\mathbb{R}^{n}$ \cite{debreu1959theory}]
    \label{theo:debreu_utility_representation}
    Let $\Omega$ be any nonempty subset of $\mathbb{R}^{n}$ and $\succsim$ be a complete preference relation on $\Omega$. If $\succsim$ is continuous, then it can be represented by a continuous utility function
    $u_{\scriptscriptstyle \succsim}:\Omega \to \mathbb{R}$ such that, $\forall \boldsymbol{x}_i, \boldsymbol{x}_j \in \Omega$:
    \begin{equation*}
        \boldsymbol{x}_i \succsim \boldsymbol{x}_j \quad \text{if and only if} \quad
        u_{\scriptscriptstyle \succsim}\left(\boldsymbol{x}_i\right) \geq u_{\scriptscriptstyle \succsim}\left(\boldsymbol{x}_j\right).
    \end{equation*}
    Moreover, we have that:
    \begin{align*}
        \begin{split}
            \boldsymbol{x}_i \succ \boldsymbol{x}_j \quad &\text{if and only if} \quad
            u_{\scriptscriptstyle \succsim}\left(\boldsymbol{x}_i\right) > u_{\scriptscriptstyle \succsim}\left(\boldsymbol{x}_j\right), \\
            \boldsymbol{x}_i \sim \boldsymbol{x}_j \quad &\text{if and only if} \quad
            u_{\scriptscriptstyle \succsim}\left(\boldsymbol{x}_i\right) = u_{\scriptscriptstyle \succsim}\left(\boldsymbol{x}_j\right).
        \end{split}
    \end{align*}
\end{theorem}

Using Theorem \ref{theo:debreu_utility_representation}, we can build an optimization
problem to find the $\succsim$-maximum of $\Omega$ as
\begin{align}
    \label{eq:preference_relation_maximum_optimization_problem}
    \boldsymbol{x}^{\boldsymbol{*}} & = \argmax{\boldsymbol{x}} u_{\scriptscriptstyle \succsim}(\boldsymbol{x}) \\
    \text{s.t.}                     & \quad\boldsymbol{x}\in\Omega, \nonumber
\end{align}
which is equivalent to Problem \eqref{eq:black_box_optimization_problem_no_constraints} by setting \mbox{$f(\boldsymbol{x}) = - u_{\scriptscriptstyle \succsim}(\boldsymbol{x})$}.
% \footnote{Formally, $f\left(\boldsymbol{x}\right)$ in the black-box framework and $u\left(\boldsymbol{x}\right)$ in the preference-based one have different domains ($\mathbb{R}^{n \times 1}$ and $\Omega$ respectively). However, assuming that $u\left(\boldsymbol{x}\right)$ is continuous and $\Omega$ is a compact subset of $\mathbb{R}^{n \times 1}$ (which are either results or assumptions of Proposition \ref{rem:existence_of_preference_relation_maximum} and Theorem \ref{theo:debreu_utility_representation}), then there exists a continuous extensions of $u\left(\boldsymbol{x}\right)$ with domain $\mathbb{R}^{n \times 1}$ (Tietze Extension Theorem \cite{kelley2017general}).}.
To avoid confusion, we refer to $f(\boldsymbol{x})$ in the preference-based framework
as the \textit{scoring function} and, similarly to the black-box case, its analytical formulation is unknown.

\begin{remark}
    Formally, $f\left(\boldsymbol{x}\right)$ in the black-box framework and $u_{\scriptscriptstyle \succsim}\left(\boldsymbol{x}\right)$ in the preference-based one have different domains
    ($\mathbb{R}^{n}$ and $\Omega$ respectively).
    However, assuming that $u_{\scriptscriptstyle \succsim}\left(\boldsymbol{x}\right)$ is continuous and $\Omega$ is a compact subset of $\mathbb{R}^{n}$
    (which are either results or assumptions of Proposition \ref{prop:existence_of_preference_relation_maximum}
    and Theorem \ref{theo:debreu_utility_representation}),
    then there exists a continuous extensions of $u_{\scriptscriptstyle \succsim}\left(\boldsymbol{x}\right)$ with domain $\mathbb{R}^{n}$ (Tietze Extension Theorem \cite{kelley2017general}).
\end{remark}

Instead of considering the preference relation explicitly (or the utility theory framework), most preference-based optimization algorithms define an (unknown) \textit{preference function} $\pi:\mathbb{R}^{n} \times \mathbb{R}^{n} \to \{-1,0,1\}$ which describes the output of the comparison between two samples.
Here, we consider $\pi\left(\boldsymbol{x}_{i},\boldsymbol{x}_{j}\right)$\footnote{
    There exist different formulations of the preference function. For example, the authors of \cite{gonzalez2017preferential} define $\pi\left(\boldsymbol{x}_{i},\boldsymbol{x}_{j}\right)$ as the probability of $\boldsymbol{x}_{i}$ being preferred over $\boldsymbol{x}_{j}$.
} as defined in \cite{Bemporad2021}:
\begin{equation}
    \label{eq:preference_function}
    \pi(\boldsymbol{x}_{i},\boldsymbol{x}_{j})=\begin{cases}
        -1 & \text{if }f\left(\boldsymbol{x}_{i}\right)<f\left(\boldsymbol{x}_{j}\right)\Longleftrightarrow\text{if }\boldsymbol{x}_{i} \succ \boldsymbol{x}_{j} \\
        0  & \text{if }f\left(\boldsymbol{x}_{i}\right)=f\left(\boldsymbol{x}_{j}\right)\Longleftrightarrow\text{if }\boldsymbol{x}_{i} \sim \boldsymbol{x}_{j}  \\
        1  & \text{if }f\left(\boldsymbol{x}_{i}\right)>f\left(\boldsymbol{x}_{j}\right)\Longleftrightarrow\text{if }\boldsymbol{x}_{j} \succ \boldsymbol{x}_{i}
    \end{cases}.
\end{equation}

The preference function \eqref{eq:preference_function} is obtained from the utility representation of the binary relation $\succsim$ (see Theorem \ref{theo:debreu_utility_representation})
and from the fact that $f(\boldsymbol{x}) = - u_{\scriptscriptstyle \succsim}(\boldsymbol{x})$.
Reflexivity and transitivity of the preorder $\succsim$ are highlighted
by the following properties of $\pi\left(\boldsymbol{x}_{i},\boldsymbol{x}_{j}\right)$:
\begin{enumerate}
    \item $\pi(\boldsymbol{x}_{i},\boldsymbol{x}_{i}) = 0, \forall \boldsymbol{x}_{i} \in \mathbb{R}^{n}$,
    \item $\pi(\boldsymbol{x}_{i},\boldsymbol{x}_{j}) = \pi(\boldsymbol{x}_{j},\boldsymbol{x}_{k}) = b \Rightarrow \pi(\boldsymbol{x}_{i},\boldsymbol{x}_{k}) = b$,
          $\forall \boldsymbol{x}_{i}, \boldsymbol{x}_{j}, \boldsymbol{x}_{k} \in \mathbb{R}^{n}$.
\end{enumerate}

In the context of preference-based optimization, surrogate-based methods aim to find the $\succsim$-maximum of $\Omega$ (by solving Problem \eqref{eq:preference_relation_maximum_optimization_problem} which is equivalent to Problem \eqref{eq:black_box_optimization_problem_no_constraints}) starting from a set of samples $\mathcal{X}$  as defined in \eqref{eq:sample_set_X}, and a set of $M$ \textit{preferences} expressed by the human decision-maker
\begin{equation}
    \label{eq:Set_of_preferences_B}
    \mathcal{B}=\left\{ b_{h}:h=1,\ldots,M,b_{h}\in\{-1,0,1\}\right\}.
\end{equation}
$b_h$ in \eqref{eq:Set_of_preferences_B} is the $h$-th preference obtained by comparing a certain couple of samples, as highlighted by the following \textit{mapping set}:
\begin{align}
    \label{eq:Mapping_set_for_preferences_S}
    \mathcal{S}=\Big\{\left(\ell(h),\kappa(h)\right): \  & h=1,\ldots,M,\ell(h),\kappa(h)\in\mathbb{N}, \nonumber                                       \\
                                                         & b_{h}=\pi\left(\boldsymbol{x}_{\ell(h)},\boldsymbol{x}_{\kappa(h)}\right), \nonumber         \\
                                                         & b_{h}\in\mathcal{B},\boldsymbol{x}_{\ell(h)},\boldsymbol{x}_{\kappa(h)}\in\mathcal{X}\Big\},
\end{align}
where $\ell: \mathbb{N} \to \mathbb{N}$ and $\kappa: \mathbb{N} \to \mathbb{N}$ are two mapping functions
that associate the indexes of the samples, contained inside $\mathcal{X}$,
to their respective preferences in $\mathcal{B}$.
This time, the cardinalities are $\left|\mathcal{X}\right|=N$
and $\left|\mathcal{B}\right|=\left|\mathcal{S}\right|=M$. Also note that
$1\leq M\leq\begin{pmatrix}N \\
        2
    \end{pmatrix}$.

Table \ref{tab:summary_table_black_box_vs_preference_based} summarizes the formulations of the black-box and preference-based optimization problems.

\begin{table*}[!ht]
    \centering
    \resizebox{\textwidth}{!}{ % Fit to page
        \def\arraystretch{1.35} % Padding verticale
        \begin{tabular}{|c|c|c|c|}
            \cline{3-4} \cline{4-4}
            \multicolumn{1}{c}{}                       &                                                                                                                            & \multicolumn{2}{c|}{\textbf{Information used to build surrogate model} $\hat{f}\left(\boldsymbol{x}\right)$}\tabularnewline
            \cline{2-4} \cline{3-4} \cline{4-4}
            \multicolumn{1}{c|}{}                      & \textbf{Problem to solve}                                                                                                  & \textbf{Set of samples}                                                                                                     & \textbf{Information on} $f\left(\boldsymbol{x}\right)$\tabularnewline
            \hline
            \multirow{1}{*}{\textbf{Black-box}}        & \multirow{1}{*}{$\boldsymbol{x}^{\boldsymbol{*}}=\argmin{\boldsymbol{x}\in\Omega}f(\boldsymbol{x})$}                       & \multirow{5}{*}{$\mathcal{X}$}                                                                                              & Function measures: $\mathcal{Y}$\tabularnewline
            \cline{1-2} \cline{2-2} \cline{4-4}
            \multirow{4}{*}{\textbf{Preference-based}} & find $\boldsymbol{x}^{\boldsymbol{*}} \in \Omega$ such that $\boldsymbol{x^{*}}\succsim\boldsymbol{x},\forall\boldsymbol{x}\in\Omega$ &                                                                                                                             & \tabularnewline
                                                       & ($\succsim$-maximum of $\Omega$)                                                                                           &                                                                                                                             & Expressed preferences: $\mathcal{B}$\tabularnewline
                                                       & $\downarrow$                                                                                                               &                                                                                                                             & Preference mapping: $\mathcal{S}$\tabularnewline
                                                       & $\boldsymbol{x}^{\boldsymbol{*}}=\argmin{\boldsymbol{x}\in\Omega}f(\boldsymbol{x})$                                        &                                                                                                                             & \tabularnewline
            \hline
        \end{tabular}
    }
    \caption{\label{tab:summary_table_black_box_vs_preference_based} Summary of the information used to solve black-box and preference-based optimization problems. }
\end{table*}
\section{Surrogate models}
\label{sec:Surrogate_models}
In the context of surrogate-based methods, a \textit{surrogate model} $\hat{f}:\mathbb{R}^{n} \to \mathbb{R}$ is an approximation of the black-box cost function or the scoring function $f\left(\boldsymbol{x}\right)$ that is (usually) inexpensive to evaluate. Its objective is to drive the optimization algorithm towards candidate samples that are minimizers of $f\left(\boldsymbol{x}\right)$. The most commonly used surrogate models are based either on Radial Basis Functions (RBFs) or Gaussian Processes (GPs). In this Section, we show how both models can be used to approximate either the black-box cost function or the scoring function.

\subsection{Surrogates based on Radial Basis Functions}
\label{subsec:Surrogate_models_RBF}
In this case, the surrogate model is defined by a \textit{radial basis function expansion} \cite{fasshauer2007meshfree} as
\begin{align}
    \label{eq:RBF_surrogate_model}
    \hat{f}\left(\boldsymbol{x}\right) & = \sum_{i=1}^{N}\beta^{(i)}\cdot\varphi\left(\epsilon\cdot\euclideannorm{\boldsymbol{x}-\boldsymbol{x}_{i}}\right) \\
                                       & = \boldsymbol{\phi}\left(\boldsymbol{x}\right)^{\top}\cdot\boldsymbol{\beta}, \nonumber
\end{align}
where $\varphi:\mathbb{R}_{\geq 0}\to\mathbb{R}$ is a properly chosen \textit{radial function} \cite{fornberg2015primer}, $\boldsymbol{\phi}\left(\boldsymbol{x}\right) \in \mathbb{R}^{N}$ is the \textit{radial basis function vector},
\begin{equation*}
    \boldsymbol{\phi}\left(\boldsymbol{x}\right)=
    \begin{bmatrix}
        \varphi\big(\epsilon\cdot\euclideannorm{\boldsymbol{x}-\boldsymbol{x}_{1}}\big) &
        \ldots                                                                             &
        \varphi\big(\epsilon\cdot\euclideannorm{\boldsymbol{x}-\boldsymbol{x}_{N}}\big)
    \end{bmatrix}^{\top},
\end{equation*}
$\epsilon \in \mathbb{R}_{>0}$ is the so-called \textit{shape parameter} (which needs to be tuned) and
$\boldsymbol{\beta}=\begin{bmatrix}\beta^{(1)} & \ldots & \beta^{(N)}\end{bmatrix}^{\top}\in\mathbb{R}^{N}$
is a {vector of weights} that has to be computed from data at hand.

\subsubsection{Black-box optimization}
In the context of black-box optimization, especially if the measures of
$f\left(\boldsymbol{x}\right)$ in $\mathcal{Y}$ are noiseless, it is desirable to have a surrogate model that interpolates the given points. That is because, as the number of samples increases,
$\hat{f}(\boldsymbol{x})$ gets sufficiently expressive to capture where the
global minimizer of $f\left( \boldsymbol{x} \right)$ is located \cite{jones2001taxonomy}.
To do so, we enforce the \textit{interpolation conditions} and calculate $\boldsymbol{\beta}$ in \eqref{eq:RBF_surrogate_model}
by solving the following linear system:
\begin{equation}
    \label{eq:Beta_computation_RBF_black_box_optimization_linear_system}
    \Phi\cdot\boldsymbol{\beta} =\boldsymbol{y},
\end{equation}
where $\Phi\in\mathbb{R}^{N\times N}$ is a symmetric matrix whose $(i,j)$-th element is
$\Phi^{(i,j)} = \varphi\left(\epsilon\cdot\euclideannorm{\boldsymbol{x}_{i} -\boldsymbol{x}_{j}}\right)$
and $\boldsymbol{y}\in\mathbb{R}^{N}$ is a vector which contains the entries of set $\mathcal{Y}$, i.e.
\mbox{$\boldsymbol{y} = \begin{bmatrix} y_1 & \ldots & y_N \end{bmatrix}^{\top}$}. 
 
The matrix $\Phi$ might be singular depending on the choice of the radial function and on the points contained in $\mathcal{X}$ \cite{gutmann2001radial}.
Moreover, the shape parameter $\epsilon$ as well as the number and the distribution of the samples $\boldsymbol{x}_i \in \mathcal{X}$ affect the condition number of $\Phi$ \cite{fasshauer2007meshfree,rippa1999algorithm}.
In \cite{Bemporad2020}, the authors propose to solve the linear system in 
\eqref{eq:Beta_computation_RBF_black_box_optimization_linear_system} using a {low-rank approximation}
of $\Phi$.
% , computed by neglecting its singular values that are below a threshold $\epsilon_\text{SVD} \in \mathbb{R}_{>0}$.
% In practice, if we proceed this way, $\hat{f}\left(\boldsymbol{x}\right)$ might not interpolate the given samples. Nonetheless, using a low-rank approximation of $\Phi$ can prove to be particularly useful if the data is noisy \cite{Bemporad2020}. 
Alternatively, a polynomial function of a certain degree can be added to the surrogate model  \eqref{eq:RBF_surrogate_model}, ensuring the existence of a unique interpolant \cite{gutmann2001radial}.

\subsubsection{Preference-based optimization}
In the context of preference-based optimization, a surrogate model $\hat{f}\left(\boldsymbol{x}\right)$ can be used to define the \textit{surrogate preference function} $\hat{\pi}:{\mathbb{R}^{n}\times\mathbb{R}^{n}\to\{-1,0,1\}}$. Differently from $\pi(\boldsymbol{x}_{i},\boldsymbol{x}_{j})$ in \eqref{eq:preference_function}, we consider a tolerance $\sigma \in \mathbb{R}_{>0}$ to avoid using strict inequalities and equalities and define $\hat{\pi}(\boldsymbol{x}_{i},\boldsymbol{x}_{j})$ as \cite{Bemporad2021}:
\begin{equation}
    \label{eq:RBF_surrogate_preference_function}
    \hat{\pi}(\boldsymbol{x}_{i},\boldsymbol{x}_{j})=\begin{cases}
        -1 & \text{if }\hat{f}\left(\boldsymbol{x}_{i}\right)-\hat{f}\left(\boldsymbol{x}_{j}\right)\leq-\sigma             \\
        0  & \text{if }\left|\hat{f}\left(\boldsymbol{x}_{i}\right)-\hat{f}\left(\boldsymbol{x}_{j}\right)\right|\leq\sigma \\
        1  & \text{if }\hat{f}\left(\boldsymbol{x}_{i}\right)-\hat{f}\left(\boldsymbol{x}_{j}\right)\geq\sigma
    \end{cases}.
\end{equation}
Instead of enforcing the interpolation conditions, \textit{we are interested in a surrogate preference function that correctly describes the preferences expressed in $\mathcal{B}$ and $\mathcal{S}$}.
This, in turn, translates into constraints on the surrogate model $\hat{f}\left(\boldsymbol{x}\right)$, which can be used to find $\boldsymbol{\beta}$ in  \eqref{eq:RBF_surrogate_model}. In order to do so, the authors of \cite{Bemporad2021} define
the following optimization problem:
\begin{align}
    \label{eq:Beta_computation_RBF_preference_based_optimization}
                & \quad \argmin{\boldsymbol{\varepsilon},\boldsymbol{\beta}}
    \frac{\lambda}{2} \cdot \boldsymbol{\beta}^{\top} \cdot \boldsymbol{\beta} + \boldsymbol{g}^\top \cdot \boldsymbol{\varepsilon}                                                                                                                                                   \\
    \text{s.t.} & \quad\hat{f}\left(\boldsymbol{x}_{\ell(h)}\right)-\hat{f}\left(\boldsymbol{x}_{\kappa(h)}\right)\leq-\sigma+\varepsilon^{(h)}                                                                                                        & \forall h:b_{h}=-1 \nonumber \\
                & \quad\left|\hat{f}\left(\boldsymbol{x}_{\ell(h)}\right)-\hat{f}\left(\boldsymbol{x}_{\kappa(h)}\right)\right|\leq\sigma+\varepsilon^{(h)}                                                                                            & \forall h:b_{h}=0  \nonumber \\
                & \quad\hat{f}\left(\boldsymbol{x}_{\ell(h)}\right)-\hat{f}\left(\boldsymbol{x}_{\kappa(h)}\right)\geq\sigma-\varepsilon^{(h)}                                                                                                         & \forall h:b_{h}=1 \nonumber  \\
                & \quad\boldsymbol{\varepsilon}\geq\boldsymbol{0}_{M}                                                                                                                                                                 \nonumber                                \\
                & \quad h=1,\ldots,M, \nonumber
\end{align}
where $\boldsymbol{\varepsilon} = \begin{bmatrix} \varepsilon^{(1)} & \ldots & \varepsilon^{(M)} \end{bmatrix}^{\top} \in\mathbb{R}^{M}$ is a vector of \textit{slack variables} (one for each preference), 
$\boldsymbol{g} = \begin{bmatrix} g^{(1)} & \ldots & g^{(M)} \end{bmatrix}^{\top} \in \mathbb{R}^{M}_{>0}$ 
is a vector of weights and $\lambda \in \mathbb{R}_{\geq 0}$ plays the role of a \textit{regularization parameter}. 

Problem \eqref{eq:Beta_computation_RBF_preference_based_optimization} ensures that, at least approximately, $\hat{f}\left(\boldsymbol{x}\right)$ is a suitable representation of the unknown preference relation $\succsim$ which generated the data (see Theorem \ref{theo:debreu_utility_representation}). 
The slacks $\boldsymbol{\varepsilon}$ are added because the surrogate model might not be complex enough to describe the given preferences, or in case some of them are expressed inconsistently by the individual. 
In practice, Problem \eqref{eq:Beta_computation_RBF_preference_based_optimization}
can be employed for any choice of $\hat{f}\left(\boldsymbol{x}\right)$ which depends upon some parameters vector $\boldsymbol{\beta}$. 
If the surrogate model is linear in $\boldsymbol{\beta}$ (such as the one in \eqref{eq:RBF_surrogate_model}), 
then Problem \eqref{eq:Beta_computation_RBF_preference_based_optimization} is a convex Quadratic Program (QP) for $\lambda > 0$ 
or a Linear Program (LP) for $\lambda = 0$ \cite{Bemporad2021}.

\subsection{Surrogates based on Gaussian Processes}
\label{subsec:Surrogate_models_GP}
In this case, we impose a Gaussian Process (GP) \cite{rasmussen2006gaussian} \textit{prior distribution} on the unknown cost function as
\begin{equation}
    \label{eq:GP_prior_on_cost_function}
    f\left(\boldsymbol{x}\right)\sim\mathcal{GP}\left(0,k\left(\boldsymbol{x}_i,\boldsymbol{x}_j\right)\right).
\end{equation}
The mean of the GP is assumed to be the zero function and $k\left(\boldsymbol{x}_i,\boldsymbol{x}_j\right)$ is a suitable \textit{kernel}\footnote{The radial functions $\varphi\left(\epsilon\cdot\euclideannorm{\boldsymbol{x}_i -\boldsymbol{x}_{j}}\right)$ used in  \eqref{eq:RBF_surrogate_model} are suitable kernels.} (or {covariance function}) which  
possibly depends on some {hyperparameters}.
%(similarly to the shape parameter $\epsilon$ for $\varphi(\cdot)$ in Eq. \eqref{eq:RBF_surrogate_model}). 
Under the assumption in \eqref{eq:GP_prior_on_cost_function}, the probability associated to the latent values $\boldsymbol{f} = \begin{bmatrix} f\left(\boldsymbol{x}_1\right) & \ldots & f\left(\boldsymbol{x}_N\right)  \end{bmatrix}^{\top} \in \mathbb{R}^{N}$
assumed by $f\left(\boldsymbol{x}\right)$ at the sampled points in $\mathcal{X}$ is
\begin{equation}
    \label{eq:GP_prior_on_f}
    p\left(\boldsymbol{f}\right) = \mathcal{N}\left(\boldsymbol{0}_{N},K\right),
\end{equation}
where $K\in\mathbb{R}^{N\times N}$ is a symmetric matrix whose $(i,j)$-th entry is $K^{(i,j)} = k\left(\boldsymbol{x}_i, \boldsymbol{x}_j \right)$ and $\mathcal{N}\left(\boldsymbol{\mu}_{\mathcal{N}}, \Sigma_\mathcal{N}\right)$ is a Gaussian distribution with mean $\boldsymbol{\mu}_{\mathcal{N}}$ and covariance $\Sigma_\mathcal{N}$.

Based on the specific optimization framework, a suitable \textit{likelihood} which describes the dataset at hand, i.e. either $\mathcal{Y}$ or $\mathcal{B}$ and $\mathcal{S}$, needs to be defined. Here, we will denote the likelihood as $p\left(\mathcal{D}|\boldsymbol{f}, \mathcal{X}\right)$, where $\mathcal{D}$ indicates a generic dataset, containing either the function measures or the preferences. Once $p\left(\mathcal{D}|\boldsymbol{f}, \mathcal{X}\right)$ has been defined, it is possible to marginalize it with respect to $\boldsymbol{f}$ to obtain the \textit{marginal likelihood} $p\left(\mathcal{D}|\mathcal{X}\right)$. The latter is often used to recalibrate the hyperparameters of the kernel \cite{rasmussen2006gaussian}. 
Finally, using Bayes' Theorem \cite{bishop2006pattern}, we can calculate the \textit{posterior distribution} $p\left(\boldsymbol{f}|\mathcal{D}, \mathcal{X}\right)$ and, more importantly, the
\textit{predictive distribution} $p\left(\tilde{f}|\mathcal{D}, \mathcal{X}, \boldsymbol{\tilde{x}}\right)$, $\tilde{f} = f\left(\boldsymbol{\tilde{x}}\right)$, whose mean can be used as the surrogate model $\hat{f}\left(\boldsymbol{x}\right)$.

\subsubsection{Black-box optimization}
In the black-box framework, given that $y_{i}=f\left(\boldsymbol{x}_{i}\right) + \eta_i$ as reported in  \eqref{eq:function_measure_set_Y} and $\eta_i$ is a realization of a Gaussian white noise,  the likelihood is 
\begin{equation}
    \label{eq:GP_likelihood_black-box}
    p\left(\mathcal{D}|\boldsymbol{f}, \mathcal{X}\right) = \mathcal{N}\left(\boldsymbol{f},\sigma_{\eta}^{2}\cdot I_{N\times N}\right),
\end{equation}
where $I_{N \times N}$ is the $N \times N$ identity matrix.
Using the properties of Gaussian distributions \cite{rasmussen2006gaussian}, it is possible to compute the expression of the predictive distribution in closed form  as
\begin{subequations}
\begin{align}
    \label{eq:GP_predictive_distribution_black-box}
    &p\left(\tilde{f}|\mathcal{D}, \mathcal{X}, \boldsymbol{\tilde{x}}\right) =
    \mathcal{N}
    \left(
        \mu_{\tilde{f}}, 
		\Sigma_{\tilde{f}}
    \right),  \\
\mu_{\tilde{f}} &= \boldsymbol{k}\left(\boldsymbol{\tilde{x}}\right)^{\top}\cdot\left[K + \sigma_{\eta}^{2}\cdot I_{N\times N}\right]^{-1}\cdot\boldsymbol{y}, \\
\Sigma_{\tilde{f}} &=  k\left(\boldsymbol{\tilde{x}},\boldsymbol{\tilde{x}}\right)-\boldsymbol{k}\left(\boldsymbol{\tilde{x}}\right)^{\top} \left[K+\sigma_{\eta}^{2} \cdot I_{N\times N}\right]^{-1} \boldsymbol{k}\left(\boldsymbol{\tilde{x}}\right),
\end{align}
\end{subequations}
where
$\boldsymbol{k}\left(\boldsymbol{\tilde{x}}\right) =
\begin{bmatrix}
    k\left(\boldsymbol{x}_1, \boldsymbol{\tilde{x}} \right) &
    \ldots                                                                             &
    k\left(\boldsymbol{x}_N, \boldsymbol{\tilde{x}} \right)
\end{bmatrix}^{\top}
\in \mathbb{R}^{N}
$
is the \textit{kernel vector}. The surrogate model is the expected value of the predictive distribution in \eqref{eq:GP_predictive_distribution_black-box}, which can be written as
\begin{equation}
    \label{eq:GP_surrogate_model}
    \hat{f}\left(\boldsymbol{x}\right) = \boldsymbol{k}\left(\boldsymbol{x}\right)^{\top}\cdot\boldsymbol{\beta}
\end{equation}

Notice that \eqref{eq:GP_surrogate_model} it is quite similar to  \eqref{eq:RBF_surrogate_model}, but this time $\boldsymbol{\beta}=\left[K+\sigma_{\eta}^{2}\cdot I_{N\times N}\right]^{-1}\boldsymbol{y}$.

In practice, $\sigma^2_\eta$ is unknown and needs to be estimated from data.
% (for example, it can be found together with the hyperparameters of the kernel by maximizing the marginal likelihood). Notice that,
If data is assumed to be noiseless ($\sigma^2_\eta = 0$) then, provided that $K$ is nonsingular, $\hat{f}\left(\boldsymbol{x}\right)$ in \eqref{eq:GP_surrogate_model} interpolates the samples in $\mathcal{X}$ and $\mathcal{Y}$.

\subsubsection{Preference-based optimization}
Gaussian Processes have also been employed in the context of preference learning and preference-based optimization. A widely used likelihood is proposed in \cite{chu2005preference}, 
where the authors only consider the strict preference relation $\succ$ instead of $\succsim$ (the indifference relation $\sim$ is not handled explicitly).
Under this assumption, it is possible to define the mapping functions $\ell(h),\kappa(h)$ in \eqref{eq:Mapping_set_for_preferences_S} so that 
$$\boldsymbol{x}_{\ell(h)} \succ \boldsymbol{x}_{\kappa(h)}, \quad \forall h = 1, \ldots, M, $$
making the set $\mathcal{B}$ in \eqref{eq:Set_of_preferences_B} redundant.
Additionally, the scoring function $f\left(\boldsymbol{x}\right)$ is assumed to be
affected by a Gaussian white noise noise 
$\eta_{\ell\left(h\right)} \overset{i.i.d.}{\sim} \mathcal{N}\left(0, \sigma^2_\eta\right)$, i.e. 
$$y_{\ell\left(h\right)} = f \left(\boldsymbol{x}_{\ell(h)}\right) + \eta_{\ell\left(h\right)},$$
and similarly for those values indexed by $\kappa(h)$.
Then, $\boldsymbol{x}_{\ell(h)} \succ \boldsymbol{x}_{\kappa(h)}$ whenever 
$y_{\ell\left(h\right)} < y_{\kappa\left(h\right)}$.
The  noise is used to capture possible inconsistencies in the preferences expressed by the individual 
(similarly to the role of the slacks in Problem \eqref{eq:Beta_computation_RBF_preference_based_optimization}).
The likelihood proposed in \cite{chu2005preference} reads as
\begin{align}
    \label{eq:GP_likelihood_preference-based}
    p\left(\mathcal{D}|\boldsymbol{f}, \mathcal{X}\right) &= \prod_{h=1}^{M} p\left(y_{\ell\left(h\right)}<y_{\kappa\left(h\right)}|f_{\ell\left(h\right)},f_{\kappa\left(h\right)}, \boldsymbol{x}_{\ell\left(h\right)}, \boldsymbol{x}_{\kappa\left(h\right)} \right) \nonumber\\
    &=\prod_{h=1}^{M}\Phi_{\mathcal{N}}\left(\frac{f_{\kappa\left(h\right)}-f_{\ell\left(h\right)}}{\sqrt{2}\cdot\sigma_{\eta}}\right), 
\end{align}
where $\Phi_{\mathcal{N}}\left(\cdot\right)$ is the standard cumulative normal distribution.
In this case, it is not possible to obtain the posterior distribution in closed form. Instead, 
the authors of \cite{chu2005preference} resort to its Laplace Approximation \cite{bishop2006pattern},
which requires solving an additional optimization problem to find the \textit{Maximum A Posteriori} (MAP)
estimate of the latent function values, $\boldsymbol{f}_{MAP} \in \mathbb{R}^{n}$.
In particular, 
$$p\left(\boldsymbol{f}| \mathcal{D}, \mathcal{X}\right) \approx 
\mathcal{N}\left(\boldsymbol{f}_{MAP}, \left[K^{-1}+\Lambda_{MAP}\right]^{-1}\right),$$ 
where $\Lambda_{MAP} \in \mathbb{R}^{N \times N}$ is the Hessian of the negative log likelihood 
$- \ln p\left(\mathcal{D}|\boldsymbol{f}, \mathcal{X}\right)$ evaluated at $\boldsymbol{f}_{MAP}$. 
Finally,  the predictive distribution can be obtained using the Laplace Approximation of the posterior distribution:
\begin{subequations}
\begin{align}
    \label{eq:GP_predictive_distribution_preference-based}
    &p\left(\tilde{f}|\mathcal{D}, \mathcal{X}, \boldsymbol{\tilde{x}}\right) =
    \mathcal{N}
    \left(
        \mu_{\tilde{f}}, 
        \Sigma_{\tilde{f}}   
    \right),  \\
\mu_{\tilde{f}} &= \boldsymbol{k}\left(\boldsymbol{\tilde{x}}\right)^{\top}\cdot K^{-1}\cdot\boldsymbol{f}_{MAP}, \\
\Sigma_{\tilde{f}}  &=     
k\left(\boldsymbol{\tilde{x}},\boldsymbol{\tilde{x}}\right)-\boldsymbol{k}\left(\boldsymbol{\tilde{x}}\right)^{\top}\cdot\left[K+\Lambda_{MAP}^{-1}\right]^{-1}\cdot\boldsymbol{k}\left(\boldsymbol{\tilde{x}}\right).
\end{align}
\end{subequations}
The surrogate model is the expected value of the predictive distribution, which can be written as in \eqref{eq:GP_surrogate_model} with $\boldsymbol{\beta} = K^{-1}\cdot\boldsymbol{f}_{MAP}$.

\section{Handling exploration and exploitation}
\label{sec:Handling_exploration_and_exploitation}
As previously mentioned, surrogate-based methods iteratively propose new samples to try with the aim of solving Problem \eqref{eq:black_box_optimization_problem_no_constraints}, while also minimizing the number of costly evaluations/comparisons. Suppose that, \textit{at iteration $k$}, we have at our disposal the set of samples $\mathcal{X}$,  $\left|\mathcal{X}\right| = N$, and either set $\mathcal{Y}$ or sets $\mathcal{B}$ and $\mathcal{S}$. We denote the best sample found so far by the procedure (i.e. the one that either achieved the lowest function value or that is preferred by the user) as
\begin{align*}
     & \boldsymbol{x_{best}}\left(N\right) \in \mathbb{R}^{n}, \boldsymbol{x_{best}}\left(N\right) \in \mathcal{X}, \left| \mathcal{X}\right| = N, & \text{such that} \\
     & \boldsymbol{x_{best}}\left(N\right) = \argmin{\boldsymbol{x}_i \in \mathcal{X}} y_i, y_i \in \mathcal{Y}    & \text{or}        \\
     & \boldsymbol{x_{best}}\left(N\right) \succsim \boldsymbol{x}_i, \forall \boldsymbol{x}_i \in \mathcal{X}.    &
\end{align*}
The new candidate sample,
$$\boldsymbol{x}_{N+1} \in \mathbb{R}^{n}, \boldsymbol{x}_{N+1} \notin \mathcal{X},$$
is obtained by solving an additional optimization problem:
\begin{align}
    \label{eq:Next_sample_search_no_black-box_constraints}
    \boldsymbol{x}_{N+1} & = \argmin{\boldsymbol{x}} a(\boldsymbol{x}) \\
    \text{s.t.}          & \quad\boldsymbol{x}\in\Omega, \nonumber
\end{align}
where $a: \mathbb{R}^{n} \to \mathbb{R}$ is a properly defined \textit{acquisition function} which trades off exploration and exploitation.
Once $\boldsymbol{x}_{N+1}$ has been computed:
\begin{itemize}
    \item In the black-box optimization case, we measure the black-box function
          at the new sample, obtaining $y_{N+1} = f\left(\boldsymbol{x}_{N+1}\right) + \eta_{N+1}$;
    \item In the preference-based framework, we let the user express a preference
          between the best sample found so far and the new one, obtaining
          $b_{M+1} = \pi\left(\boldsymbol{x}_{N+1},\boldsymbol{x_{best}}\left(N\right)\right)$.
\end{itemize}
In both cases, $\boldsymbol{x}_{N+1}$ is added to the set $\mathcal{X}$ and,
similarly, $\mathcal{Y}, \mathcal{B}$ and $\mathcal{S}$ are also updated
with either $y_{N+1}$ or $b_{M+1}$. The
process is iterated until a certain condition is met. Usually, a \textit{budget}, or
rather a maximum number of samples to evaluate $N_{max}$, is set and the procedure
is stopped once it is reached.

In this work, $a\left(\boldsymbol{x}\right)$ is defined starting from a surrogate model  and an \textit{exploration function} $z: \mathbb{R}^{n} \to \mathbb{R}$ which leads the optimization procedure towards regions of $\Omega$ where few samples have been tried and/or where the surrogate model is most uncertain. We assume that both $\hat{f}(\boldsymbol{x})$ and $z(\boldsymbol{x})$ are continuous functions.
The acquisition function that we adopt here is an explicit trade-off between these two functions:
\begin{equation}
    \label{eq:Acquisition_function_no_black_box_constraints_v1}
    a(\boldsymbol{x}) = \delta\cdot\frac{\hat{f}(\boldsymbol{x}) - \hat{f}_{min}\left( \mathcal{X}_{aug}\right)}{\Delta\hat{F}\left( \mathcal{X}_{aug}\right)}
    +\left(1 - \delta\right)\cdot\frac{z(\boldsymbol{x}) - z_{min}\left( \mathcal{X}_{aug}\right)}{\Delta Z\left( \mathcal{X}_{aug}\right)},
\end{equation}
where:
\begin{itemize}
    \item $\delta \in \left[0, 1\right]$ is a parameter which defines the \textit{exploration-exploitation trade-off}.
    \item $\hat{f}\left(\boldsymbol{x}\right)$ and $z\left(\boldsymbol{x}\right)$ have  been rescaled using \textit{min-max normalization} \cite{han2011data} in order to make them assume the same range $\left[0, 1\right]$ (or, at least, make them comparable).
          In particular, given any function $h:\mathbb{R}^{n} \to \mathbb{R}$ and a set of  samples $\mathcal{X}_{aug}$, we define
          \begin{subequations}
              \begin{align}
                  h_{min}\left(\mathcal{X}_{aug}\right)  & = \min_{\boldsymbol{x} \in \mathcal{X}_{aug}}h(\boldsymbol{x}),                \\
                  h_{max}\left(\mathcal{X}_{aug}\right)  & = \max_{\boldsymbol{x} \in \mathcal{X}_{aug}}h(\boldsymbol{x}),                \\
                  \Delta H\left(\mathcal{X}_{aug}\right) & = h_{max}\left(\mathcal{X}_{aug}\right)-h_{min}\left(\mathcal{X}_{aug}\right).
              \end{align}
          \end{subequations}
          Note that, to avoid dividing by zero in \eqref{eq:Acquisition_function_no_black_box_constraints_v1}, $\Delta H\left(\mathcal{X}_{aug}\right)$ can be set to $h_{max}\left(\mathcal{X}_{aug}\right)$ or $1$ whenever
          $h_{min}\left(\mathcal{X}_{aug}\right) = h_{max}\left(\mathcal{X}_{aug}\right) \neq 0$ or
          $h_{min}\left(\mathcal{X}_{aug}\right) = h_{max}\left(\mathcal{X}_{aug}\right) = 0$ respectively.

    \item $\mathcal{X}_{aug} = \left\{\boldsymbol{x_{aug}}_{i}: i = 1, \ldots, N_{aug}, \boldsymbol{x}_{{aug}_i} \in \Omega \right\}$ is the so called \textit{augmented sample set}, which needs to be defined so that
          \begin{subequations}
              \begin{align}
                  \hat{f}_{min}\left(\mathcal{X}_{aug}\right) & \approx \min_{\boldsymbol{x} \in \Omega} \hat{f}(\boldsymbol{x}), \\
                  \hat{f}_{max}\left(\mathcal{X}_{aug}\right) & \approx \max_{\boldsymbol{x} \in \Omega} \hat{f}(\boldsymbol{x}), \\
                  z_{min}\left(\mathcal{X}_{aug}\right)       & \approx \min_{\boldsymbol{x} \in \Omega} z(\boldsymbol{x}),       \\
                  z_{max}\left(\mathcal{X}_{aug}\right)       & \approx \max_{\boldsymbol{x} \in \Omega} z(\boldsymbol{x}). 
              \end{align}
          \end{subequations}
          In practice, this means that $\mathcal{X}_{aug}$ needs to be sufficiently expressive to allow for a proper comparison between the surrogate model and the exploration function in \eqref{eq:Acquisition_function_no_black_box_constraints_v1}. There are different ways to obtain the augmented sample set. The most accurate (and expensive) one would be to solve four additional optimization problems to find the minimizers and maximizers of $\hat{f}(\boldsymbol{x})$ and $z(\boldsymbol{x})$. Alternatively, as it has been done for \texttt{MSRS} \cite{Regis2007}, the augmented sample set can be obtained by randomly sampling $\Omega$.
          If a-priori knowledge on the stationary points of $\hat{f}(\boldsymbol{x})$ and/or $z(\boldsymbol{x})$ is available, then it can be used to build $\mathcal{X}_{aug}$, see for example \cite{previtali2022glispr}.
          Finally, a possible choice is $\mathcal{X}_{aug} = \mathcal{X}$, however it is not recommended because, as we will see in Section \ref{subsec:Exploration_functions}, $z\left(\boldsymbol{x}\right)$ is usually maximal at the sampled points. Therefore, $\mathcal{X}$ is not expressive enough to rescale the exploration function.
\end{itemize}
% 
% The acquisition function \eqref{eq:Acquisition_function_no_black_box_constraints_v1} can be seen as a generalized version of the one proposed in \texttt{MSRS} \cite{Regis2007}, where the function $z\left(\boldsymbol{x}\right)$ is fixed a-priori. Instead, the proposed acquisition \eqref{eq:Acquisition_function_no_black_box_constraints_v1} can handle any $z\left(\boldsymbol{x}\right)$.
% %

As a final note, $a\left(\boldsymbol{x}\right)$ in \eqref{eq:Acquisition_function_no_black_box_constraints_v1} is often a multimodal function. Therefore,
global optimization procedures need to be employed to solve Problem \eqref{eq:Next_sample_search_no_black-box_constraints}.
However, compared to the black-box cost function or the interaction with the individual, $a\left(\boldsymbol{x}\right)$ is cheap to evaluate and therefore we are not particularly concerned with its number of function evaluations.

\begin{remark}
    The acquisition function \eqref{eq:Acquisition_function_no_black_box_constraints_v1} can be seen as a generalized version of the one proposed in \texttt{MSRS} \cite{Regis2007}, where the function $z\left(\boldsymbol{x}\right)$ is fixed a-priori.
    Instead, the proposed $a\left(\boldsymbol{x}\right)$ in \eqref{eq:Acquisition_function_no_black_box_constraints_v1} can use any (proper) $z\left(\boldsymbol{x}\right)$.
    Moreover, \eqref{eq:Acquisition_function_no_black_box_constraints_v1} will be employed in the proposed general optimization scheme for both black-box and preference-based problems.
    Instead, algorithm \texttt{MSRS} \cite{Regis2007} deals only with black-box problems.
\end{remark}

\subsection{Exploration functions}
\label{subsec:Exploration_functions}
In Section \ref{sec:Surrogate_models}, we showed different models that can be used as surrogates for the acquisition function \eqref{eq:Acquisition_function_no_black_box_constraints_v1}.
Here, we define possible exploration functions $z\left(\boldsymbol{x}\right)$ that are suited for \eqref{eq:Acquisition_function_no_black_box_constraints_v1}.

The aim of $z\left(\boldsymbol{x}\right)$ is to drive the optimization procedure towards regions of $\Omega$ where few samples are present.
To do so, the exploration function must use the information available at the current iteration, i.e. $\mathcal{X}$ and,
possibly but not necessarily, either the measures of the cost function $\mathcal{Y}$ or the preferences in $\mathcal{B}$ and $\mathcal{S}$.
We provide the following Definition to highlight which functions $z\left(\boldsymbol{x}\right)$ are suitable to
be used as an exploration function for \eqref{eq:Acquisition_function_no_black_box_constraints_v1}.
\begin{definition}[Proper exploration function] \label{def:Proper_exploration_function}
    Suppose that $\Omega$ is a compact subset of $\mathbb{R}^{n}$.
    Then, a function  $z:\mathbb{R}^{n} \to \mathbb{R}$ is a proper exploration function if it is continuous and the solution of Problem \eqref{eq:Next_sample_search_no_black-box_constraints} with $\delta = 0$,
    or equivalently
    \begin{align}
        \label{eq:Next_sample_search_only_exploration_no_black-box_constraints}
        \boldsymbol{x}_{N+1} & = \argmin{\boldsymbol{x}} z(\boldsymbol{x}) \\
        \text{s.t.}          & \quad\boldsymbol{x}\in\Omega, \nonumber
    \end{align}
    is not already present in $\mathcal{X}$, i.e. $\boldsymbol{x}_{N+1} \notin \mathcal{X}$.
\end{definition}
Compactness of $\Omega$ and continuity of $z\left(\boldsymbol{x}\right)$ ensure that Problem
\eqref{eq:Next_sample_search_only_exploration_no_black-box_constraints} has at least one solution. If instead it has multiple solutions, then at least one of them must not be in $\mathcal{X}$.

An exploration function could also depend on the choice of the surrogate model. For instance, if $\hat{f}\left(\boldsymbol{x}\right)$ is obtained by imposing a GP prior on $f\left(\boldsymbol{x}\right)$, then we can use the \textit{negative standard deviation} of the predictive distribution as exploration function, namely
\begin{equation}
    \label{eq:Exploration_function_std_GP_black-box}
    z\left(\boldsymbol{x}\right) = - \sqrt{k\left(\boldsymbol{x},\boldsymbol{x}\right)-\boldsymbol{k}\left(\boldsymbol{x}\right)^{\top}\cdot\left[K+\sigma_{\eta}^{2}\cdot I_{N\times N}\right]^{-1}\cdot\boldsymbol{k}\left(\boldsymbol{x}\right)}
\end{equation}
in the black-box case \eqref{eq:GP_predictive_distribution_black-box} and
\begin{equation}
    \label{eq:Exploration_function_std_GP_preference-based}
    z\left(\boldsymbol{x}\right) = - \sqrt{k\left(\boldsymbol{x},\boldsymbol{x}\right)-\boldsymbol{k}\left(\boldsymbol{x}\right)^{\top}\cdot\left[K+\Lambda_{MAP}^{-1}\right]^{-1}\cdot\boldsymbol{k}\left(\boldsymbol{x}\right)}
\end{equation}
in the preference-based one \eqref{eq:GP_predictive_distribution_preference-based}.

The functions $z\left(\boldsymbol{x}\right)$ in \eqref{eq:Exploration_function_std_GP_black-box}
and in \eqref{eq:Exploration_function_std_GP_preference-based} are continuous if
the chosen kernel function $k\left(\cdot, \boldsymbol{x}\right)$ is continuous. Moreover, the variance of the predictive distribution is \textit{minimal} at the sampled values in $\mathcal{X}$ \cite{rasmussen2006gaussian},
% 
% \authornote{(In realtà non so bene che cosa citare o come dimostrarlo)}
while it assumes higher values where the surrogate model is most uncertain.
Therefore, $z\left(\boldsymbol{x}\right)$ in \eqref{eq:Exploration_function_std_GP_black-box} and in   \eqref{eq:Exploration_function_std_GP_preference-based} are proper exploration functions.

Alternative exploration functions that are not related to the surrogate model $\hat{f}\left(\boldsymbol{x}\right)$ exist. For example, the authors of \texttt{GLIS} \cite{Bemporad2020} proposed the \textit{Inverse Distance Weighting} (IDW) \textit{distance function}:
\begin{equation}
    \label{eq:Inverse_distance_weighting_distance_function}
    z\left(\boldsymbol{x}\right)=\begin{cases}
        0                                                                                                 & \text{if }\boldsymbol{x}\in\mathcal{X} \\
        {-\frac{2}{\pi}}\cdot\arctan\left(\frac{1}{\sum_{i=1}^{N}w_{i}\left(\boldsymbol{x}\right)}\right) & \text{otherwise}
    \end{cases},
\end{equation}
where $w_{i}:\mathbb{R}^{n}\setminus\left\{ \boldsymbol{x}_{i}\right\} \to\mathbb{R}_{>0}$,
$w_{i}\left(\boldsymbol{x}\right)=\cfrac{1}{\euclideannorm{\boldsymbol{x}-\boldsymbol{x}_{i}}^{2}}$, is the IDW function \cite{shepard1968two}.
In \cite{Bemporad2020}, the authors also prove that $z\left(\boldsymbol{x}\right)$ is differentiable everywhere on $\mathbb{R}^{n}$ and hence it is continuous.
Another exploration function is the one used in \texttt{MSRS} \cite{Regis2007}:
\begin{equation}
    \label{eq:Exploration_function_MSRS}
    z\left(\boldsymbol{x}\right) = - \min_{\boldsymbol{x}_i \in \mathcal{X}} \euclideannorm{\boldsymbol{x} - \boldsymbol{x}_i},
\end{equation}
which is continuous since it is the composition of continuous functions.
Both $z\left(\boldsymbol{x}\right)$ in \eqref{eq:Inverse_distance_weighting_distance_function}
and in \eqref{eq:Exploration_function_MSRS} are zero only at $\boldsymbol{x}_i \in \mathcal{X}$ and assume negative values $\forall \boldsymbol{x} \notin \mathcal{X}$. Thus, they are proper exploration functions.

\subsection{Relationship to other surrogate-based algorithms}
\label{subsec:Relationship_with_other_algorithms}
Often, acquisition functions based on explicit trade-offs between a surrogate
model and an exploration function exhibit the following structure:
\begin{equation}
    \label{eq:general_acquisition_function_explicit_trade-off}
    a\left(\boldsymbol{x}\right) = \hat{f}\left(\boldsymbol{x}\right) +
    \alpha \cdot z\left(\boldsymbol{x}\right),
\end{equation}
where $\alpha \in \mathbb{R}$ is a suitable coefficient that can be varied in between iterations of the optimization procedure.
The proposed acquisition function \eqref{eq:Acquisition_function_no_black_box_constraints_v1} belongs to this rationale. It is possible to prove that, for $\delta \neq 0$, \eqref{eq:Acquisition_function_no_black_box_constraints_v1} has the same minimizer as:
\begin{equation}
    \label{eq:Acquisition_function_no_black_box_constraints_v2}
    a(\boldsymbol{x})= \hat{f}(\boldsymbol{x}) +
    \frac{1-\delta}{\delta} \cdot \frac{\Delta\hat{F}\left(\mathcal{X}_{aug}\right)}{\Delta Z\left(\mathcal{X}_{aug}\right)} \cdot z(\boldsymbol{x}).
\end{equation}
For some specific choices of $\hat{f}\left(\boldsymbol{x}\right)$ and $z\left(\boldsymbol{x}\right)$, the proposed acquisition functions $a\left(\boldsymbol{x}\right)$ in \eqref{eq:Acquisition_function_no_black_box_constraints_v1}
or \eqref{eq:Acquisition_function_no_black_box_constraints_v2} can be seen as a generalization of the acquisition functions used by some other popular surrogate-based methods, like

\begin{enumerate}
    \item \texttt{MSRS} \cite{Regis2007} is a black-box optimization algorithm which uses the same acquisition function \eqref{eq:Acquisition_function_no_black_box_constraints_v1}, does not make any assumption on the surrogate model
          $\hat{f}\left(\boldsymbol{x}\right)$, and adopts the exploration function \eqref{eq:Exploration_function_MSRS}.
          Moreover, the points in $\mathcal{X}_{aug}$ are generated randomly and, instead of explicitly solving Problem \eqref{eq:Next_sample_search_no_black-box_constraints},
          the new candidate sample is selected as
          $$\boldsymbol{x}_{N+1} = \argmin{\boldsymbol{x} \in \mathcal{X}_{aug}}
              a\left( \boldsymbol{x} \right).$$
    \item In the context of Bayesian Optimization, a popular acquisition function is the so called \textit{Lower Confidence Bound} (often referred to as \texttt{GP-LCB}) \cite{brochu2010tutorial},
          which can be obtained by using the acquisition function \eqref{eq:general_acquisition_function_explicit_trade-off} with $\hat{f}\left(\boldsymbol{x}\right)$
          defined as in Section \ref{subsec:Surrogate_models_GP} and $z\left(\boldsymbol{x}\right)$ as \eqref{eq:Exploration_function_std_GP_black-box}
          or \eqref{eq:Exploration_function_std_GP_preference-based}, depending on the optimization framework.
          In practice, $\alpha$ in \eqref{eq:general_acquisition_function_explicit_trade-off}
          for \texttt{GP-LCB} \cite{brochu2010tutorial} is often kept
          constant throughout the whole optimization procedure.
    \item In the preference-based framework, algorithm \texttt{GLISp} \cite{Bemporad2021} uses a RBF surrogate model \eqref{eq:RBF_surrogate_model}
          and $z\left(\boldsymbol{x}\right)$ as in \eqref{eq:Inverse_distance_weighting_distance_function}.
          Its acquisition function is defined as
          $$a(\boldsymbol{x})=\frac{\hat{f}\left(\boldsymbol{x}\right)}{\Delta\hat{F}\left(\mathcal{X}\right)}+\alpha\cdot z\left(\boldsymbol{x}\right),$$
          which has the same minimizer as the one in  \eqref{eq:Acquisition_function_no_black_box_constraints_v2}
          for $\mathcal{X}_{aug} = \mathcal{X}$ and a proper choice of $\delta$.
          % 
          % In \authornote{Paper 2}, we extend \texttt{GLISp} following the scheme proposed in Section \ref{sec:General_optimization_scheme}. In particular, the augmented sample set is built using some information on the exploration function in \eqref{eq:Inverse_distance_weighting_distance_function}.
\end{enumerate}

\subsection{Choosing the trade-off parameter}
\label{subsec:Greedy_delta_cycling}
Many black-box optimization algorithms {explicitly vary the exploration-exploitation trade-off in between the iterations of the procedure}. Just to cite a few:
\begin{itemize}
    \item \texttt{Gutmann-RBF} \cite{gutmann2001radial} uses an acquisition function that is a measure of ``bumpiness'' of the RBF surrogate, which depends upon a target value $t$ to aim for.
          % This value needs to be chosen appropriately and it must be such that $-\infty < t \leq \min_{\boldsymbol{x} \in \Omega} \hat{f}\left(\boldsymbol{x} \right)$. Choosing $t = \min_{\boldsymbol{x} \in \Omega} \hat{f}\left(\boldsymbol{x} \right)$ means trusting the surrogate model and thus $\boldsymbol{x}_{N+1}$ will be close to the minimizer of $\hat{f}\left(\boldsymbol{x}\right)$ (local search). Viceversa, setting $t = -\infty$ results in looking for $\boldsymbol{x}_{N+1}$ in a region of the domain that has not yet been explored (global search). Moreover, the author has proven that choosing $t = \min_{\boldsymbol{x} \in \Omega} \hat{f}\left(\boldsymbol{x} \right)$ at each iteration does not ensure the convergence of the algorithm. 
          The values of $t$ are cycled between   two extrema   to alternate between local and global search.
    \item The authors of \texttt{MSRS} \cite{Regis2007}, which uses the acquisition function   \eqref{eq:Acquisition_function_no_black_box_constraints_v1} with $z\left(\boldsymbol{x}\right)$ as in \eqref{eq:Exploration_function_MSRS},  propose to cycle between different values of $\delta$ as to prioritize exploration or exploitation more.
    \item In algorithm \texttt{SO-SA} \cite{wang2014general}, which is a revisitation of \texttt{MSRS} \cite{Regis2007}, the weight $\delta$ is chosen in a random fashion at each iteration. Moreover, the authors adopt a greedy strategy, i.e. the trade-off is kept unaltered until it fails to find a significantly better solution.
          %, claiming that it helps in the case of high dimensional problems with a very limited number of function evaluations.
    \item In the context of Bayesian optimization, a popular way to find the next candidate sample is to maximize the \textit{Probability of Improvement}, which is defined as
          $$
              p\left(f\left(\boldsymbol{x}\right) \leq \hat{f}_{min}\left(\mathcal{X}\right) - \xi\right) =
              \Phi_{\mathcal{N}}\left(\frac{\hat{f}_{min}\left( \mathcal{X}\right) - \xi - \hat{f}\left(\boldsymbol{x}\right)}{-z\left(\boldsymbol{x}\right)}\right). \nonumber
          $$
          where
          $\hat{f}\left(\boldsymbol{x}\right)$ and $-z\left(\boldsymbol{x}\right)$ are the mean and the standard deviation of the predictive distribution (see Section \ref{subsec:Surrogate_models_GP}),
          while $\xi \in \mathbb{R}_{\geq 0}$ is a trade-off parameter that needs to be tuned.
          In \cite{kushner1964new}, $\xi$ is initialized to a high value so that the algorithm
          prioritizes exploration in the early iterations and gets progressively smaller to give more importance to the surrogate later on.
\end{itemize}
In this work, we use the \textit{greedy $\delta$-cycling} strategy, which we proposed in \cite{previtali2022glispr} and we now briefly review.
We define a set of $N_{cycle} \geq 1$ weights to cycle:
\begin{equation}
    \label{eq:Delta_cycle}
    \Delta_{cycle} = \left\{\delta_0, \ldots, \delta_{N_{cycle - 1}}\right\}.
\end{equation}
The set $\Delta_{cycle}$ should contain values that are well spread within the $\left[0, 1\right]$ range as to properly alternate between local and global search.
Then, as long as $\boldsymbol{x_{best}}\left(N\right)$ varies from an iteration to the other (i.e. there has been some improvement), hyperparameter $\delta$ in \eqref{eq:Acquisition_function_no_black_box_constraints_v1} is kept unchanged.
Viceversa, whenever the algorithm produces an $\boldsymbol{x}_{N+1}$ that is not better than the best sample found so far $\boldsymbol{x_{best}}\left(N\right)$, the weight is cycled following the order proposed in $\Delta_{cycle}$. More formally, suppose that, at iteration $k$, we have at our disposal $\left|\mathcal{X}\right| = N$ samples and denote the trade-off parameter $\delta$ in \eqref{eq:Acquisition_function_no_black_box_constraints_v1} as $\delta\left(k\right)$ to highlight the iteration number. Furthermore, assume $\delta \left(k\right) = \delta_j \in \Delta_{cycle}$, which has been used to find the new candidate sample $\boldsymbol{x}_{N+1}$ at iteration $k$ by solving Problem \eqref{eq:Next_sample_search_no_black-box_constraints}. Then, at iteration $k + 1$, we select $\delta \left(k+1\right) \in \Delta_{cycle}$ as:
\begin{equation*}
    \delta \left(k + 1\right) = \begin{cases}
        \delta_j & \text{If $\boldsymbol{x_{best}}\left(N + 1\right) = \boldsymbol{x}_{ N+ 1}$} \\
        \delta_{\left(j + 1\right)\text{mod}{N_{cycle}}} & \text{If $\boldsymbol{x_{best}}\left(N + 1\right) = \boldsymbol{x_{best}}\left(N\right)$} 
    \end{cases}
\end{equation*}
The convergence of the optimization scheme that we propose in the next Section is strictly
related to the choice of the cycling set \eqref{eq:Delta_cycle}.

\begin{comment}
\begin{algorithm}[h]
    \setstretch{1.35}
    \caption{Greedy $\delta$-cycling}
    \label{alg:Greedy_delta_cycling}
    \textbf{Input}:
    \begin{enumerate}[(i)]
        \item Exploration-exploitation trade-off cycle $\Delta_{cycle}$ in \eqref{eq:Delta_cycle},
        \item Exploration-exploitation trade-off $\delta_i \in \Delta_{cycle}$ used at the last iteration,
        \item Flag $hasImproved$ which indicates whether or not $\boldsymbol{x_{best}\left(N\right)}$ has been replaced at the last iteration (improvement).
    \end{enumerate}
    \textbf{Output}:
    \begin{enumerate}[(i)]
        \item Parameter $\delta$ in \eqref{eq:Acquisition_function_no_black_box_constraints_v1} to use at the current iteration.
    \end{enumerate}
    \hrule
    \begin{algorithmic}[1]
        \If{$hasImproved$}
        \State $\delta = \delta_i$
        \Else
        \State $\delta = \delta_{\left(i+1\right)\mod{N_{cycle}}}$
        \EndIf
    \end{algorithmic}
\end{algorithm}
\end{comment}

\section{General optimization scheme and convergence}
\label{sec:General_optimization_scheme}
Algorithm \ref{alg:General_surrogate-based_scheme} describes a general procedure that can be used to solve Problem \eqref{eq:black_box_optimization_problem_no_constraints}, either in the black-box or preference-based framework.
We will refer to the proposed scheme as \textit{generalized Metric Response Surface} (\texttt{gMRS} for short) since it can be seen as an extension of the \texttt{MSRS} \cite{Regis2007} procedure.
Differently from \texttt{MSRS} \cite{Regis2007}, \texttt{gMRS}
\textit{can handle both optimization frameworks and different exploration functions}.

As with any surrogate-based method, \texttt{gMRS} starts from an initial set of samples $\mathcal{X}$ that needs to be generated using a suitable \textit{space-filling experimental design} \cite{vu2017surrogate}, for example
Latin Hypercube Designs (LHDs) \cite{mckay2000comparison}.
Then, the samples in $\mathcal{X}$ are evaluated either by measuring the value of the black-box cost function $f\left(\boldsymbol{x}\right)$ or by asking the individual to compare them. In any case, the initial best sample $\boldsymbol{x_{best}}\left(N \right)$ is obtained, either as the one that achieved the lowest $y_i \in \mathcal{Y}$ or by properly guiding the comparisons, using the transitive property of the preference relation $\succsim$ (see Section \ref{sec:Problems_formulation}). Iteratively, until the budget $N_{max}$ is exhausted, the surrogate model $\hat{f}\left(\boldsymbol{x}\right)$ is built (or updated) and, together with a proper exploration function $z\left(\boldsymbol{x}\right)$, used to find a new candidate sample $\boldsymbol{x}_{N+1}$ by solving Problem \eqref{eq:Next_sample_search_no_black-box_constraints}.
The sample $\boldsymbol{x}_{N+1}$, suggested by the algorithm, replaces the best sample found so far, $\boldsymbol{x_{best}}\left(N\right)$, either if
$$y_{N+1} \leq y_{best}\left(N\right),$$
where  $y_{best}\left(N\right)$ is the measure of the black-box cost function at $\boldsymbol{x_{best}}\left(N\right)$, or if
$$\boldsymbol{x}_{N+1} \succ \boldsymbol{x_{best}}\left(N\right).$$
After that, the information brought by $\boldsymbol{x}_{N+1}$ is added to the respective sets $\mathcal{X}, \mathcal{Y}, \mathcal{B}$ and $\mathcal{S}$.

Note that $\hat{f}\left(\boldsymbol{x}\right)$ possibly contains some hyperparameters that might need to be {recalibrated}.
In the case of RBF surrogates, this can be done by employing \textit{cross-validation} (see \cite{rippa1999algorithm, cavoretto2021search} for black-box optimization and \cite{Bemporad2021} for the preference-based case).
Instead, for GP surrogates, we can \textit{maximize the marginal likelihood} (see \cite{rasmussen2006gaussian} and \cite{chu2005preference} for black-box and preference-based optimization respectively). Recalibration might not be performed at every iteration but only at certain ones.

\begin{remark}
    Further algorithmic details, such as the possibility of rescaling the decision variable $\boldsymbol{x}$
    (see for example \cite{Bemporad2020}) or handling the case when $\boldsymbol{x}_{N+1}$
    returned by Problem \eqref{eq:Next_sample_search_no_black-box_constraints} has already been tried,
    i.e. $\boldsymbol{x}_{N+1} \in \mathcal{X}$ (this could happen if $\delta = 1$),
    are not covered in Algorithm \ref{alg:General_surrogate-based_scheme} but can easily be included.
\end{remark}

\begin{algorithm*}[h]
    \setstretch{1.35}
    \caption{\texttt{gMRS} optimization scheme}
    \label{alg:General_surrogate-based_scheme}
    \textbf{Input}:
    \begin{enumerate}[(i)]
        \item Constraint set $\Omega$ in \eqref{eq:constraint_set_Omega},
        \item Initial number of samples $N$ (must be greater than $2$ in the preference-based case),
        \item Budget $N_{max} > N$,
        \item Surrogate model $\hat{f}\left(\boldsymbol{x}\right)$ (Section \ref{sec:Surrogate_models}) with, possibly, its hyperparameters,
        \item Proper exploration function $z\left(\boldsymbol{x}\right)$ (Section \ref{subsec:Exploration_functions}),
        \item Exploration-exploitation trade-off cycle $\Delta_{cycle}$ (Section \ref{subsec:Greedy_delta_cycling}).
    \end{enumerate}
    \textbf{Output}:
    \begin{enumerate}[(i)]
        \item Best sample obtained by the procedure $\boldsymbol{x_{best}}\left(N_{max}\right)$.
    \end{enumerate}
    \hrule
    \begin{algorithmic}[1]
        \State Select a set of starting points $\mathcal{X}$, $\left|\mathcal{X}\right| = N$,
        using a suitable experimental design \cite{vu2017surrogate}
        \State Evaluate the samples in $\mathcal{X}$, obtaining some information on the cost function $f\left(\boldsymbol{x}\right)$
        (either set $\mathcal{Y}$ in \eqref{eq:function_measure_set_Y}
        or sets $\mathcal{B}$ and $\mathcal{S}$ in \eqref{eq:Set_of_preferences_B}
        and \eqref{eq:Mapping_set_for_preferences_S}),
        and get the initial best sample $\boldsymbol{x_{best}}\left(N\right)$
        \For{$k = 1, 2, \ldots, N_{max} - \left|\mathcal{X}\right|$}
        
        \State (Optional) Recalibrate the hyperparameters of the surrogate model
        $\hat{f}\left(\boldsymbol{x}\right)$
        \State Build or update surrogate model $\hat{f}\left(\boldsymbol{x}\right)$ from
        $\mathcal{X}$ and the information on $f\left(\boldsymbol{x}\right)$ at hand
        \State Build the augmented sample set $\mathcal{X}_{aug}$
        \State Select $\delta$ for the current iteration from $\Delta_{cycle}$ (Section \ref{subsec:Greedy_delta_cycling})
        \State Solve Problem \eqref{eq:Next_sample_search_no_black-box_constraints} to obtain
        the new candidate sample $\boldsymbol{x}_{N+1}$
        \State Either measure the value of the cost function for $\boldsymbol{x}_{N+1}$
        or let the human decision-maker express a preference between $\boldsymbol{x}_{N+1}$ and $\boldsymbol{x_{best}}\left(N\right)$
        \If{$\boldsymbol{x}_{N+1}$ achieved a better result than $\boldsymbol{x_{best}}\left(N\right)$}
        \State Set $\boldsymbol{x_{best}}\left(N + 1\right) = \boldsymbol{x}_{N+1}$
        \Else
        \State Set $\boldsymbol{x_{best}}\left(N + 1\right) = \boldsymbol{x_{best}}\left(N\right)$ (no improvement)
        \EndIf
        \State Update the set of samples $\mathcal{X}$ and either the
        collection of measures $\mathcal{Y}$ or the
        user-expressed preferences $\mathcal{B}$ and $\mathcal{S}$.
        \State Set $N = N + 1$
        \EndFor
    \end{algorithmic}
\end{algorithm*}

\subsection{Convergence of \texttt{gMRS}}
\label{subsec:Convergence}
It is possible to guarantee the convergence of any global optimization algorithm to the global minimizer of
Problem \eqref{eq:black_box_optimization_problem_no_constraints} by proving the following Theorem.
\begin{theorem}[Convergence of a global optimization algorithm
        \cite{torn1989global}]
    \label{theo:convergence_theorem_torn}
    Consider the global optimization problem in
    \eqref{eq:black_box_optimization_problem_no_constraints}.
    Let $\Omega \subset \mathbb{R}^{n}$ be a compact set and $f:
        \mathbb{R}^{n} \to \mathbb{R}$ be a continuous function.
    Then, an algorithm converges to the global minimum of every continuous
    function on $\Omega$ if and only if its sequence of iterates,
    \begin{equation*}
        \langle\boldsymbol{x}_i\rangle_{i \geq 1} = \langle\boldsymbol{x}_1, \boldsymbol{x}_2, \ldots \rangle,
    \end{equation*}
    is everywhere dense in $\Omega$.
\end{theorem}

Concerning Algorithm \ref{alg:General_surrogate-based_scheme}, we can generalize the convergence result obtained for \texttt{GLISp-r} in \cite{previtali2022glispr} to \texttt{gRMS}, as claimed by the following Theorem.
\begin{theorem}[Convergence of \texttt{gMRS}]
    \label{theo:Convergence_of_gMRS}
    Let $\Omega \subset \mathbb{R}^{n}$ be a compact set and either:
    \begin{itemize}
        \item $f:\mathbb{R}^{n} \to \mathbb{R}$ be a continuous function (black-box case) or,
        \item $\succsim$ be a continuous and complete preference relation (preference-based case).
    \end{itemize}
    If $z\left(\boldsymbol{x}\right)$ is a proper exploration function, as defined in Definition \ref{def:Proper_exploration_function}, and there $\exists \delta_j \in \Delta_{cycle}$ such that $\delta_j = 0$, then, for $N_{max} \to \infty$, \texttt{gMRS} converges to the global minimizer of Problem \eqref{eq:black_box_optimization_problem_no_constraints} for any set of initial points $\mathcal{X}$, $\left| \mathcal{X} \right| = N$, as well as any continuous surrogate model $\hat{f}\left(\boldsymbol{x}\right)$.
\end{theorem}
\begin{proof}
    In the black-box framework, continuity of $f\left(\boldsymbol{x}\right)$ and compactness of $\Omega$ ensure that there exists a global minimizer for Problem \eqref{eq:black_box_optimization_problem_no_constraints} (Extreme Value Theorem \cite{audet2017derivative}) and are required for Theorem \ref{theo:convergence_theorem_torn}. Similarly, continuity and completeness of $\succsim$ guarantee that there exists a $\succsim$-maximum of $\Omega$ for Proposition \ref{prop:existence_of_preference_relation_maximum}. Moreover, from Theorem \ref{theo:debreu_utility_representation}, there exists a continuous scoring function $f\left( \boldsymbol{x}\right)$ that represents $\succsim$ and such that solving Problem \eqref{eq:black_box_optimization_problem_no_constraints} leads to find the $\succsim$-maximum of $\Omega$. In turn, this makes it possible to apply Theorem \ref{theo:convergence_theorem_torn} also in the preference-based framework.

    Consider the sequence of iterates $\langle\boldsymbol{x}_i\rangle_{i \geq 1}$ produced by Algorithm \ref{alg:General_surrogate-based_scheme}. We define
    \begin{itemize}
        \item $\mathcal{X}_\infty$ as the set containing all the elements of $\langle\boldsymbol{x}_i\rangle_{i \geq 1}$,
        \item The subsequence of $\langle\boldsymbol{x}_i\rangle_{i \geq 1}$ containing only its first $k$ entries as $\langle\boldsymbol{x}_i\rangle_{i = 1}^k = \langle\boldsymbol{x}_1, \ldots, \boldsymbol{x}_k\rangle$,
        \item $\mathcal{X}_k$ as the collection of the points in $\langle\boldsymbol{x}_i\rangle_{i = 1}^k$.
    \end{itemize}
    In practice, the first $N$ entries of $\langle\boldsymbol{x}_i\rangle_{i \geq 1}$ constitute the initial set of samples $\mathcal{X}$ (obtained by an experimental design), i.e. $\mathcal{X}_N = \mathcal{X}$, while the remaining ones are obtained by solving Problem \eqref{eq:Next_sample_search_no_black-box_constraints}, which always admits a solution since both $\hat{f}\left(\boldsymbol{x}\right)$ and $z\left(\boldsymbol{x}\right)$ are assumed to be continuous. Any sample $\boldsymbol{x}_i$ obtained either by the experimental design or by solving Problem \eqref{eq:Next_sample_search_no_black-box_constraints} is such that $\boldsymbol{x}_i \in \Omega$, therefore $\mathcal{X}_\infty \subseteq \Omega$.

    Suppose now that $\Delta_{cycle} = \left\{0\right\}$, then, at each iteration, the new candidate sample $\boldsymbol{x}_{k+1}$ ($k > N$) is found by solving
    Problem \eqref{eq:Next_sample_search_only_exploration_no_black-box_constraints} (pure exploration) using $k$ samples (contained in $\mathcal{X}_k$).
    Since $z\left(\boldsymbol{x}\right)$ is a proper exploration function, $\boldsymbol{x}_{k+1} \notin \mathcal{X}_{k}$, which implies that, given any $\boldsymbol{x} \in \Omega$, $\boldsymbol{x} \in \mathcal{X}_k$ for $k \to \infty$.
    In other words, any point $\boldsymbol{x} \in \Omega$ will eventually be sampled by Problem \eqref{eq:Next_sample_search_only_exploration_no_black-box_constraints}, provided that $z\left(\boldsymbol{x}\right)$ is proper. Thus, we can define a sequence
    \begin{equation*}
        \langle\boldsymbol{\tilde{x}}_i\rangle_{i \geq 1} = \langle\boldsymbol{\tilde{x}}_1, \boldsymbol{\tilde{x}}_2, \ldots \rangle
    \end{equation*}
    in $\mathcal{X}_{\infty}$ as the concatenation of a  sequence $\langle\boldsymbol{x}_i\rangle_{i = 1}^k$ for $k$ such that $\boldsymbol{x}_{k} = \boldsymbol{x}$ and a constant sequence of $\boldsymbol{x}$, i.e.
    \begin{equation*}
        \langle\boldsymbol{\tilde{x}}_i\rangle_{i \geq 1} = \langle\boldsymbol{x}_1, \ldots, \boldsymbol{x}_{k-1}, \boldsymbol{x}, \boldsymbol{x}, \ldots \rangle.
    \end{equation*}
    By construction, $\langle\boldsymbol{\tilde{x}}_i\rangle_{i \geq 1}$ is such that
    \begin{equation}
        \label{eq:Limit_of_sequence_x_tilde_convergence}
        \lim_{i \to \infty} \boldsymbol{\tilde{x}}_i = \boldsymbol{x}.
    \end{equation}
    We have proven that:
    \begin{itemize}
        \item $\mathcal{X}_\infty \subseteq \Omega$,
        \item Given any $\boldsymbol{x} \in \Omega$, there exists a sequence $\langle\boldsymbol{\tilde{x}}_i\rangle_{i \geq 1}$ in $\mathcal{X}_{\infty}$ which satisfies \eqref{eq:Limit_of_sequence_x_tilde_convergence}.
    \end{itemize}
    Thus, we can conclude that $\mathcal{X}_\infty$ is dense in $\Omega$ \cite{kelley2017general} and, consequently, so is the corresponding sequence of iterates $\langle\boldsymbol{x}_i\rangle_{i \geq 1}$.
    Finally, by Theorem \ref{theo:convergence_theorem_torn}, \texttt{gMRS} converges to the global minimizer of
    Problem \eqref{eq:black_box_optimization_problem_no_constraints}.
    We can reach the same conclusion for any $\Delta_{cycle}$ that includes a zero entry.
\end{proof}

\begin{remark}
    Combining the utility theory framework \cite{ok2011real} with preference-based optimization allows us to extend Theorem \ref{theo:convergence_theorem_torn} as to cover the convergence to the $\succsim$-maximum of $\Omega$. In this case, we must ensure that the preference relation admits a continuous representation (Theorem \ref{theo:debreu_utility_representation}) and we need to guarantee that a $\succsim$-maximum of $\Omega$ exists (Proposition \ref{prop:existence_of_preference_relation_maximum}). Under these assumptions, we are able to prove the convergence of \texttt{gRMS} in the preference-based case. Instead, other preference-based algorithms often neglect a formal proof of convergence.
\end{remark}

\begin{remark}
    Theorem \ref{theo:Convergence_of_gMRS} guarantees the convergence
    of \texttt{gRMS} but does not give any indication
    on its rate. In practice, it depends on a multitude of factors, such as the choice
    of the surrogate model, exploration function and cycling set.
    Setting $\Delta_{cycle} = \left\{0\right\}$ basically results in performing exhaustive
    search \cite{audet2017derivative}, which is quite inefficient but is guaranteed to
    converge to the minimizer of Problem \eqref{eq:black_box_optimization_problem_no_constraints}
    under the assumptions of Theorem \ref{theo:Convergence_of_gMRS}.
    We suggest to use a $\Delta_{cycle}$ in \eqref{eq:Delta_cycle}
    that contains values which are well spread within the $\left[0, 1\right]$ range,
    including a zero entry to guarantee the convergence.
\end{remark}

\section{Illustrative example}
\label{sec:example}
Suppose that we want to solve the following global optimization problem:
\begin{align*}
    \boldsymbol{x}^{\boldsymbol{*}} & = \argmin{\boldsymbol{x}} f\left(\boldsymbol{x}\right) \\
    \text{s.t.}                     &  \quad \begin{bmatrix}
        -1, -1
    \end{bmatrix}^\top \leq \boldsymbol{x} \leq \begin{bmatrix}
        2, 1
    \end{bmatrix}^\top
\end{align*}
where the cost function $f\left(\boldsymbol{x}\right)$ is the \texttt{adjiman} function in \cite{jamil2013literature}, i.e. 
$$f\left(\boldsymbol{x}\right) = \cos\left(x^{(1)}\right)\cdot\sin\left(x^{(2)}\right)-\frac{x^{(1)}}{\left(x^{(2)}\right)^{2}+1}.$$
We show the performances of two algorithms that follow the \texttt{gRMS} paradigm (Algorithm \ref{alg:General_surrogate-based_scheme}), in the black-box and preference-based frameworks respectively.
For this example, we assume that, in the black-box case, we are able to measure $f\left(\boldsymbol{x}\right)$ without noise ($\sigma_\eta^2 = 0$). We approximate $f\left(\boldsymbol{x} \right)$ using surrogate model \eqref{eq:RBF_surrogate_model} with $\boldsymbol{\beta}$ computed as in \texttt{GLIS} \cite{Bemporad2020}.
Viceversa, in the preference-based framework, we use the preference function in \eqref{eq:preference_function} to compare different samples. We still use $\hat{f}\left(\boldsymbol{x}\right)$ in \eqref{eq:RBF_surrogate_model} but find $\boldsymbol{\beta}$ by solving Problem \eqref{eq:Beta_computation_RBF_preference_based_optimization}, as it is done for \texttt{GLISp} \cite{Bemporad2021}. In both cases, we use $z\left(\boldsymbol{x}\right)$ in \eqref{eq:Inverse_distance_weighting_distance_function} and define the augmented sample set $\mathcal{X}_{aug}$ using some information on the stationary points of the chosen exploration function, as proposed in \cite{previtali2022glispr}. Moreover, we adopt the same cycling set $\Delta_{cycle} = \left\{0.95, 0.7, 0.35, 0\right\}$ for black-box and preference-based optimization.

We compare the previously described instances of \texttt{gRMS} to \texttt{GLIS} \cite{Bemporad2020} and \texttt{GLISp} \cite{Bemporad2021} since they both use the same surrogate models and exploration functions\footnote{Formally, \texttt{GLIS} \cite{Bemporad2020} uses an additional exploration function $s\left(\boldsymbol{x}\right)$, called the IDW variance function, and thus its acquisition function is defined as a weighted sum between $\hat{f}\left(\boldsymbol{x}\right)$, $z\left(\boldsymbol{x}\right)$ and $s\left(\boldsymbol{x}\right)$} but employ different acquisition functions. For this reason, we refer to them as \texttt{GLIS-r} and \texttt{GLISp-r}, where the \texttt{r} highlights the min-max rescaling performed in \eqref{eq:Acquisition_function_no_black_box_constraints_v1}. We use the same hyperparameters for the surrogates of \texttt{GLIS} \cite{Bemporad2020} and \texttt{GLIS-r}, as well as \texttt{GLISp} \cite{Bemporad2021} and \texttt{GLISp-r} (see \cite{previtali2022glispr} for a more formal definition of this algorithm), and set them to the values proposed in their respective papers. The remaining hyperparameters for \texttt{GLIS} \cite{Bemporad2020} and \texttt{GLISp} \cite{Bemporad2021} are selected as suggested by the authors. We remark that, in the original methods, no cycling is performed for their respective exploration-exploitation trade-off parameters. We perform $N_{MC} = 100$ Monte Carlo simulations starting from different sets of samples and with budget $N_{max} = 70$. Moreover, in the black-box framework we start from $4$ samples while in the preference-based one we begin from $8$ samples and $7$ preferences. The initial sample set $\mathcal{X}$ is generated using a Latin Hypercube Design \cite{mckay2000comparison}.
Figure \ref{fig:example} depicts the results of the Monte Carlo simulations. In the black-box framework, \texttt{GLIS} \cite{Bemporad2020} and \texttt{GLIS-r} exhibit similar performances (same convergence speed). Instead, in the preference-based case, median-wise \texttt{GLISp} \cite{Bemporad2021} finds the global minimizer slightly faster compared to \texttt{GLISp-r} but can get stuck on a local minima (as highlighted by its worst-case performances), see \cite{previtali2022glispr} for a more in-depth look. Viceversa, cycling $\delta$ in \eqref{eq:Acquisition_function_no_black_box_constraints_v1} as proposed in Section \ref{subsec:Greedy_delta_cycling} leads \texttt{GLISp-r} to converge to $\boldsymbol{x^*}$ on all Monte Carlo simulations.
\begin{figure}[!ht]
    \centering
    \subfloat{
        \centering
        \includegraphics[width=.5\textwidth]{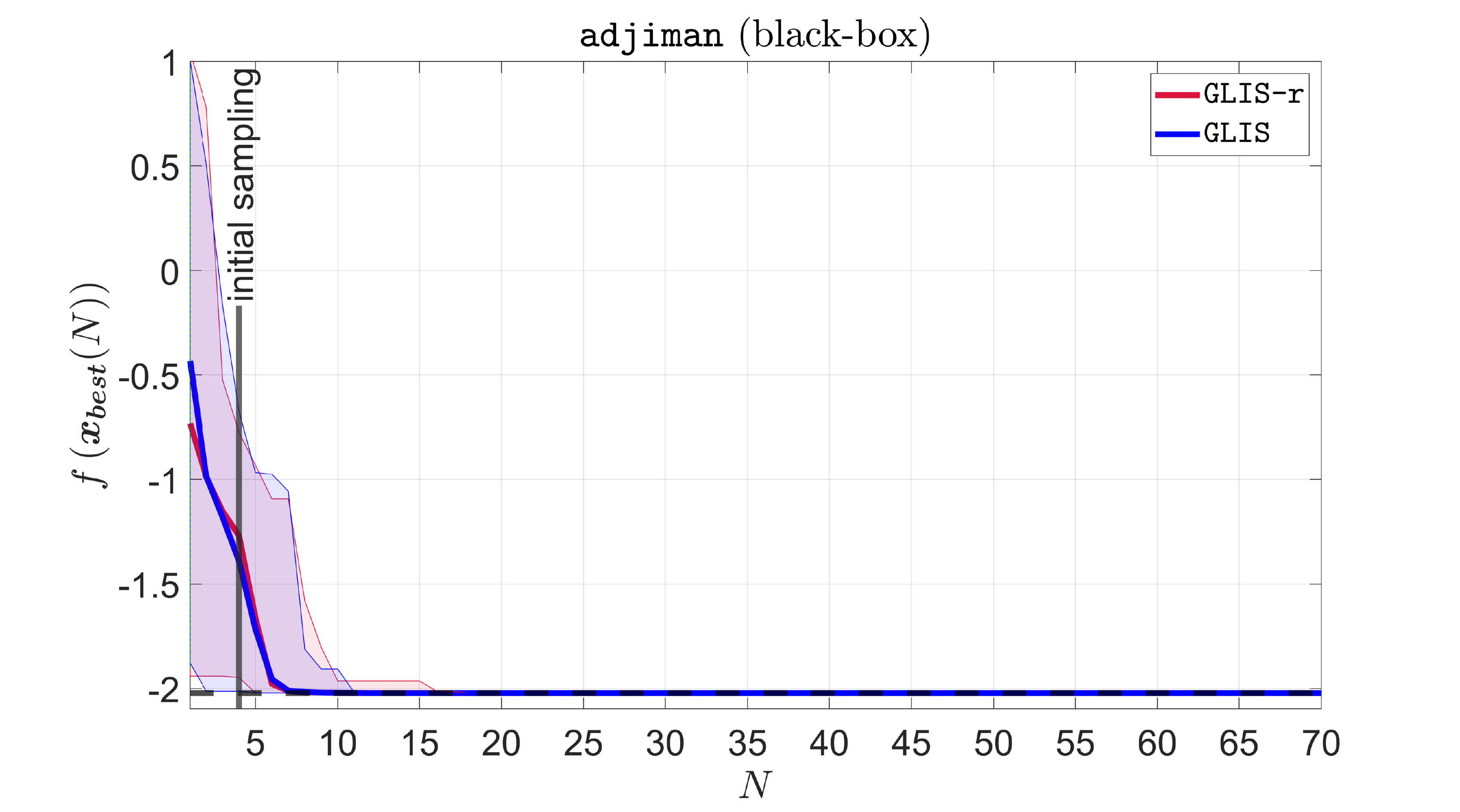}
    }
    \subfloat{
        \centering
        \includegraphics[width=.5\textwidth]{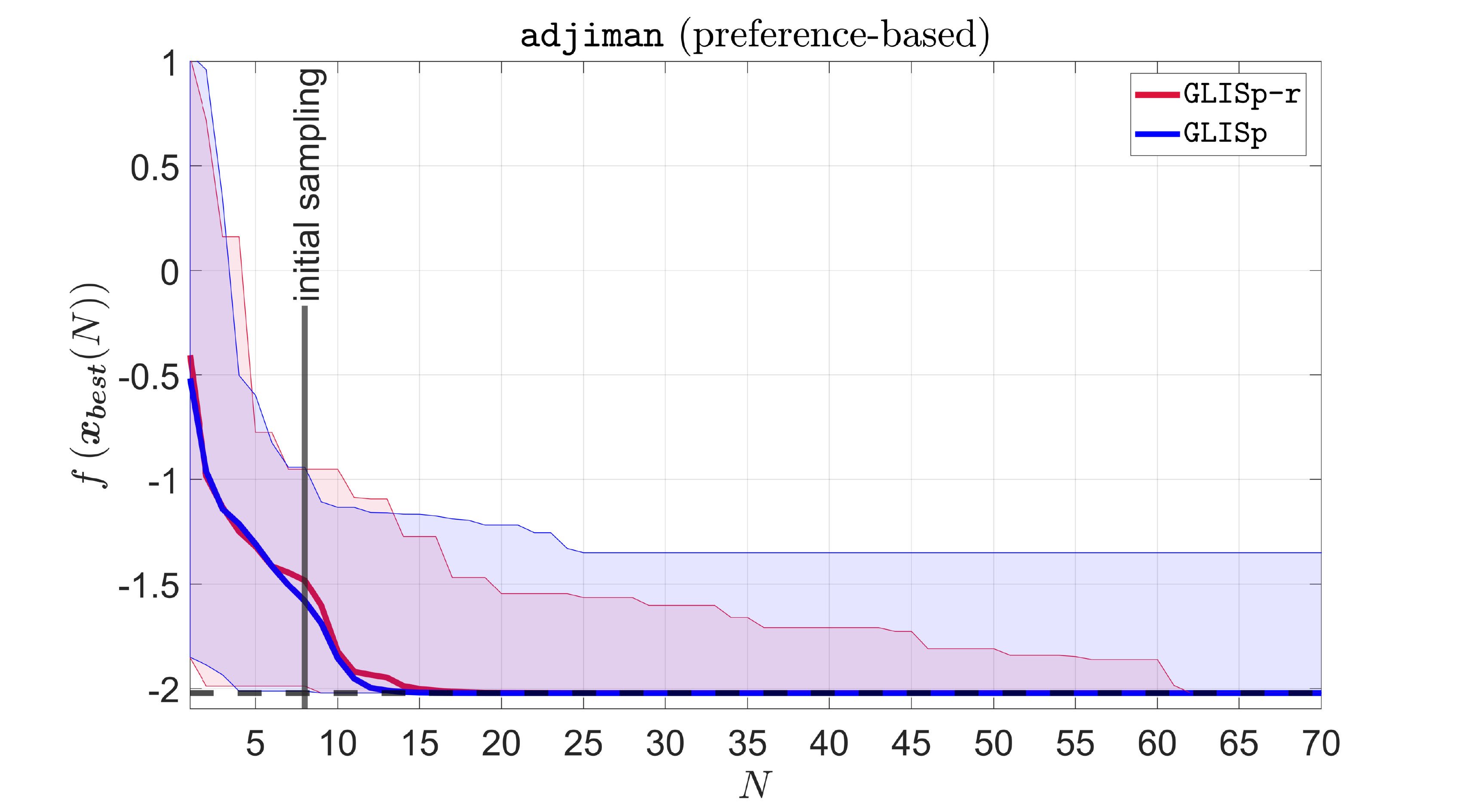}
    }

    \caption{
        \label{fig:example}
        Performance comparison between the different algorithms in the black-box and preference-based frameworks. The thick colored lines denote the median value, the shadowed areas remark the best and worst case instances, the dashed black line is the global minimum $f\left(\boldsymbol{x^*}\right)$ and the black vertical line divides the initial sampling phase and the one based on the minimization of the acquisition function.
    }
\end{figure}
\section{Conclusions}
\label{sec:Summary_and_discussion}
In this paper, we have thoroughly analyzed and compared the black-box and the preference-based optimization frameworks. Using utility theory, we have shown that, if $\succsim$ associated to the individual's criterion is a continuous and complete preference relation, then both black-box and preference-based algorithms aim to solve the same problem, that is Problem \eqref{eq:black_box_optimization_problem_no_constraints}. The only difference is the information available of the latent $f\left(\boldsymbol{x}\right)$.
We  focused our attention on surrogate-based methods, which approximate $f\left(\boldsymbol{x}\right)$ using only the data at hand. Then, we proposed a general acquisition function $a\left(\boldsymbol{x}\right)$ in  \eqref{eq:Acquisition_function_no_black_box_constraints_v1}, which is an explicit trade-off between the surrogate model $\hat{f}\left(\boldsymbol{x}\right)$ and a proper exploration function $z\left(\boldsymbol{x}\right)$, and shown how it relates to the ones used by other popular surface response methods. After that, we formalized \texttt{gMRS} (Algorithm \ref{alg:General_surrogate-based_scheme}), a general optimization scheme that can be used both in the black-box and preference-based frameworks. Its convergence is guaranteed provided that the chosen exploration function is a proper one and $\Delta_{cycle}$ includes at least a zero entry.

\begin{comment}
Further research is devoted to extending \texttt{gMRS} to handle the Problem:
\begin{align}
    \label{eq:black_box_optimization_problem_with_constraints}
    \boldsymbol{x}^{\boldsymbol{*}} & = \argmin{\boldsymbol{x}} f(\boldsymbol{x})       \\
    \text{s.t.}                     & \quad\boldsymbol{x}\in \Omega \cap \Xi, \nonumber
\end{align}
where $\Xi \subseteq \mathbb{R}^{n}$ is a set of black-box constraints (a-priori unknown).
In general, Problem \eqref{eq:black_box_optimization_problem_with_constraints} is tackled by building an additional surrogate model for the constraints in $\Xi$,
see for example \cite{zhu2021cglisp, gelbart2014bayesian}, and modifying either the acquisition function
in \eqref{eq:Acquisition_function_no_black_box_constraints_v1} or Problem
\eqref{eq:Next_sample_search_no_black-box_constraints} to penalize sampling inside the unknown unfeasible region.
\end{comment}

% Bibliography
\bibliographystyle{plain}
\bibliography{References_books, References_Bayesian_papers, References_other_papers, References_RBF_papers}

% Appendix
% \appendix

\end{document}